\documentclass[a4paper,11pt]{amsart}
\title[Orlov/Kuznetsov equivalence]{Gepner type stability condition via 
Orlov/Kuznetsov equivalence}
\date{}
\author{Yukinobu Toda}

\usepackage{amscd}
\usepackage{amsmath}
\usepackage{amssymb}
\usepackage{amsthm}
\usepackage{float}
\usepackage[dvips]{graphicx}

\usepackage[all,ps,dvips]{xy}

\usepackage{array}
\usepackage{amscd}
\usepackage[all]{xy}

\DeclareFontFamily{U}{rsfs}{%
\skewchar\font127}
\DeclareFontShape{U}{rsfs}{m}{n}{%
<-6>rsfs5<6-8.5>rsfs7<8.5->rsfs10}{}
\DeclareSymbolFont{rsfs}{U}{rsfs}{m}{n}
\DeclareSymbolFontAlphabet
{\mathrsfs}{rsfs}
\DeclareRobustCommand*\rsfs{%
\@fontswitch\relax\mathrsfs}

\theoremstyle{plain}
\newtheorem{thm}{Theorem}[section]
\newtheorem{prop}[thm]{Proposition}
\newtheorem{lem}[thm]{Lemma}

\newtheorem{defi}[thm]{Definition}
\newtheorem{rmk}[thm]{Remark}
\newtheorem{cor}[thm]{Corollary}

\newtheorem{step}{Step}
\newtheorem{sstep}{Step}

\newtheorem{prop-defi}[thm]{Proposition-Definition}
\newtheorem{thm-defi}[thm]{Theorem-Definition}
\newtheorem{lem-defi}[thm]{Lemma-Definition}

\newtheorem{conj}[thm]{Conjecture}

\newdimen\argwidth
\def\db[#1\db]{
 \setbox0=\hbox{$#1$}\argwidth=\wd0
 \setbox0=\hbox{$\left[\box0\right]$}
  \advance\argwidth by -\wd0
 \left[\kern.3\argwidth\box0 \kern.3\argwidth\right]}

\newcommand{\aA}{\mathcal{A}}
\newcommand{\bB}{\mathcal{B}}
\newcommand{\cC}{\mathcal{C}}
\newcommand{\dD}{\mathcal{D}}
\newcommand{\eE}{\mathcal{E}}
\newcommand{\fF}{\mathcal{F}}

\newcommand{\hH}{\mathcal{H}}

\newcommand{\lL}{\mathcal{L}}

\newcommand{\oO}{\mathcal{O}}
\newcommand{\pP}{\mathcal{P}}

\newcommand{\sS}{\mathcal{S}}
\newcommand{\tT}{\mathcal{T}}
\newcommand{\uU}{\mathcal{U}}

\newcommand{\zZ}{\mathcal{Z}}
\newcommand{\lr}{\longrightarrow}

\newcommand{\Hom}{\mathop{\rm Hom}\nolimits}

\newcommand{\dR}{\mathbf{R}}
\newcommand{\dL}{\mathbf{L}}

\newcommand{\Pic}{\mathop{\rm Pic}\nolimits}

\newcommand{\id}{\textrm{id}}

\newcommand{\ch}{\mathop{\rm ch}\nolimits}

\newcommand{\td}{\mathop{\rm td}\nolimits}
\newcommand{\Ext}{\mathop{\rm Ext}\nolimits}

\newcommand{\rank}{\mathop{\rm rank}\nolimits}
\newcommand{\Coh}{\mathop{\rm Coh}\nolimits}

\newcommand{\cneq}{\mathrel{\raise.095ex\hbox{:}\mkern-4.2mu=}}
\newcommand{\eqcn}{\mathrel{=\mkern-4.5mu\raise.095ex\hbox{:}}}

\newcommand{\Aut}{\mathop{\rm Aut}\nolimits}

\newcommand{\Stab}{\mathop{\rm Stab}\nolimits}

\newcommand{\Imm}{\mathop{\rm Im}\nolimits}

\newcommand{\HMF}{\mathop{\rm HMF}^{\rm{gr}}\nolimits}

\newcommand{\Ree}{\mathop{\rm Re}\nolimits}

\newcommand{\wP}{\widetilde{\mathbb{P}}}
\newcommand{\wX}{\widetilde{X}}
\newcommand{\Pd}{\partial'}
\begin{document}
\maketitle

\begin{abstract}
We show the existence of 
Gepner type Bridgeland stability conditions
on the triangulated categories of graded matrix factorizations
associated with homogeneous polynomials which 
define general cubic fourfolds containing a plane. 
The key ingredient is to describe the grade shift functor
of matrix factorizations in terms of sheaves of 
Clifford algebras on the projective plane under 
Orlov/Kuznetsov equivalence. 
\end{abstract}

\setcounter{tocdepth}{1}
\tableofcontents

\section{Introduction}
\subsection{Motivation and results}
This paper is a continuation of the previous papers~\cite{TGep}, \cite{TGep2}.
For a homogeneous polynomial
\begin{align*}
W \in A \cneq \mathbb{C}[x_1, x_2, \cdots, x_n]
\end{align*} 
of degree $d$, 
let $\HMF(W)$ be the triangulated category of graded 
matrix factorizations of $W$. 
It consists of objects given by 
data
\begin{align*}
P^0 \stackrel{p^0}{\to} P^1 
\stackrel{p^1}{\to} P^0(d)
\end{align*}
where $P^i$ are graded free $A$-modules of finite rank, 
$p^i$ are homomorphisms of graded $A$-modules,
$P^i \mapsto P^i(1)$ is the shift of the grading,
satisfying $p^1 \circ p^0 = p^0 \circ p^1 = \cdot W$. 
The category $\HMF(W)$ is considered to be 
the category of B-branes of
a graded Landau-Ginzburg model with 
superpotential $W$, and 
studying Bridgeland stability conditions~\cite{Brs1} on it 
is an interesting problem both in mathematics and 
string theory. 
Among them, we focus on a 
specific type of a stability condition, 
which has a symmetric property with 
respect to the grade shift functor 
$\tau$ on $\HMF(W)$ sending 
$P^{\bullet}$ to $P^{\bullet}(1)$.  
The existence of such specific type of
stability conditions is conjectural, and formulated as follows
(cf.~\cite[Conjecture~1.2]{TGep}):
\begin{conj}\label{intro:conj}
There is a Bridgeland stability condition
\begin{align*}
\sigma_G=(Z_G, \{\pP_G(\phi)\}_{\phi \in \mathbb{R}})
\in  \Stab(\HMF(W))
\end{align*}
where the central charge $Z_G$ is given by 
\begin{align}\notag
Z_G(P^{\bullet})= \mathrm{str}(e^{2\pi \sqrt{-1}/d} \colon 
P^{\bullet} \to P^{\bullet}) 
\end{align}
and the set of semistable objects satisfy $\tau \pP_G(\phi)=\pP_G(\phi +2/d)$. 
\end{conj}
A Bridgeland stability condition 
in Conjecture~\ref{intro:conj}
was called \textit{Gepner type} with respect to $(\tau, 2/d)$
in~\cite{TGep}, since if $W$ is a quintic polynomial with five 
variables such a stability condition is expected to correspond to the 
Gepner point in the stringy K$\ddot{\rm{a}}$hler moduli 
space of the quintic 3-fold defined by $W$, 
under mirror symmetry and Orlov's result relating $\HMF(W)$ 
with the derived category of coherent sheaves on the 
hypersurface 
\begin{align}\label{intro:X}
X \cneq (W=0) \subset \mathbb{P}^{n-1}. 
\end{align}
While the above conjecture 
is motivated by string theory~\cite{Wal}, 
this is also an interesting mathematical 
problem as the resulting 
$\sigma_G$ is 
an analogue of a Gieseker stability for 
coherent sheaves on polarized varieties. 
The Donaldson (Thomas) type theory counting $\sigma_G$-semistable 
graded matrix factorizations 
should be an analogue of Fan-Jarvis-Ruan-Witten theory~\cite{FJRW1}
in Gromov-Witten theory. 
The symmetric property of $\sigma_G$ may yield
 an interesting
relationship among the classical Donaldson (Thomas) type invariants
on the hypersurface (\ref{intro:X})
via the above mentioned Orlov's result~\cite{Orsin}
together with wall-crossing arguments~\cite{JS}, \cite{K-S}. 
On the other hand,
 constructing $\sigma_G$ in Conjecture~\ref{intro:conj}
turns out to be
a hard problem, due to the absence 
of a natural t-structure on $\HMF(W)$.
 So far Conjecture~\ref{intro:conj} is proved in the following cases: 
$n=1$~\cite{Tak}, $d<n=3$~\cite{KST1}, 
$n\le d\le 4$~\cite{TGep}, and some other 
weighted cases~\cite{KST1}, \cite{TGep}. 
The strategy in~\cite{TGep} was 
to apply Orlov's result~\cite{Orsin}
and construct desired stability conditions 
in the geometric side.

Let us focus on the low degree cases
of Conjecture~\ref{intro:conj}. 
It is almost trivial to prove it in the 
$d\le 2$ cases for any $n$ (cf.~Remark~\ref{rmk:low}), so  
the $d=3$ case is the non-trivial 
lowest degree case.
The purpose of this paper is to prove
Conjecture~\ref{intro:conj} for one of the $d=3$ 
cases in whch 
there is an interesting geometric background:
that is 
when $X$ is a cubic fourfold.  
It has been long observed that the geometry of 
cubic fourfolds is very similar to that of 
K3 surfaces~\cite{BeDo}, \cite{Voi}.  
That observation inspired some 
conjectures relating 
the rationality of cubic fourfolds 
with the existence of the corresponding K3 surfaces:
the Hodge theoretic one is due to Hassett~\cite{BH2}, 
the derived categorical one is due to Kuznetsov~\cite{Kuz2}, 
and an equivalence of these conjectures (under a 
certain genericity condition)
is due to 
Addington-Thomas~\cite{AT}. 
The main result of this 
paper, presented as follows, 
is an application of the above 
relationship between 
cubic fourfolds and K3 surfaces
to the study of Conjecture~\ref{intro:conj}:
\begin{thm}\label{intro:mainthm}
Conjecture~\ref{intro:conj} is true 
when $(d, n)=(3, 6)$ and 
the hypersurface $(W=0) \subset \mathbb{P}^5$
is a general cubic fourfold containing a plane. 
\end{thm}
Our strategy is to
combine Orlov's work~\cite{Orsin}
relating $\HMF(W)$ with $D^b \Coh(X)$, 
with Kuznetsov's work~\cite{Kuz2} relating 
the latter category with 
the derived category of
coherent sheaves on 
a 
twisted K3 surface.  
The result of Theorem~\ref{intro:mainthm}
immediately follows from the 
corresponding result for the twisted
K3 surfaces, given in Theorem~\ref{thm:introGepner} below. 
Our genericity condition and 
 more detail on the proof of Theorem~\ref{intro:mainthm}
will be given in the next subsection. 

Here we should mention the case of cubic polynomials 
with lower number of variables, i.e. $d=3$ and $n\le 4$. 
In this case,  Conjecture~\ref{intro:conj} is easier to prove
than Theorem~\ref{intro:mainthm}, 
and we give 
 some detail in Appendix~B: 
\begin{thm}\emph{(Theorem~\ref{thm:surface}, Theorem~\ref{thm:3fold}.)}
\label{intro:thm}
Conjecture~\ref{intro:conj} is true 
when $d=3$ and $n\le 5$. 
\end{thm}

\subsection{Strategy for Theorem~\ref{intro:mainthm} via Orlov/Kuznetsov equivalence}
Let 
\begin{align*}
X=(W=0) \subset \mathbb{P}^5
\end{align*}
 be a smooth 
cubic fourfold containing a plane $P$. 
The bounded derived category of coherent 
sheaves on $X$ has the following semiorthogonal 
decomposition
\begin{align*}
D^b \Coh(X)= \langle \dD_X, \oO_X, \oO_X(1), \oO_X(2) \rangle. 
\end{align*}
There are two ways to relate the semiorthogonal 
summand $\dD_X$ with other triangulated categories. 
The first one is due to Orlov~\cite{Orsin}, which 
provides an equivalence
\begin{align*}
\Phi_1 \colon \HMF(W) \stackrel{\sim}{\to} \dD_X. 
\end{align*}
The second one is due to Kuznetsov~\cite{Kuz2}, which 
provides an equivalence
\begin{align*}
\Theta \colon 
D^b \Coh(\bB_0) \stackrel{\sim}{\to} \dD_X. 
\end{align*}
Here $\bB_0$ is an even part of a 
sheaf of Clifford algebras on $\mathbb{P}^2$, 
which is constructed in~\cite{Kuz}
from a quadric fibration $\widetilde{X} \to \mathbb{P}^2$
for a blow-up $\widetilde{X} \to X$ at $P$. 
The construction also defines an odd part $\bB_1$, which 
is $\bB_0$-bimodule, and other $\bB_i$ are defined by
$\bB_{i+2}=\bB_i(1)$. 
Our first step is to relate the grade shift functor $\tau$
with the autequivalence $F_B$ of $D^b \Coh(\bB_0)$, given as
\begin{align*}
F_B \cneq \mathrm{ST}_{\bB_1}^{-1} \circ \otimes_{\bB_0}\bB_{-1}[1].
\end{align*}
Here $\mathrm{ST}_{\bB_1}$ is the Seidel-Thomas twist~\cite{ST}
associated to $\bB_1$. 
We show the following proposition:
\begin{prop}\emph{(Corollary~\ref{cor:shift}.)}\label{intro:shift}
The following diagram commutes:
\begin{align*}
\xymatrix{
D^b \Coh(\bB_0) \ar[r]^{\Phi_1^{-1} \circ 
\Theta} \ar[d]_{F_B} & \HMF(W) \ar[d]^{\tau} \\
D^b \Coh(\bB_0) \ar[r]^{\Phi_1^{-1} \circ \Theta} & \HMF(W).
}
\end{align*}
\end{prop}
The above proposition is proved 
by combining Ballard-Favero-Katzarkov's
work~\cite{BFK} describing 
$\tau$ in terms of
$\dD_X$,
with an explicit computation of $\Theta(\bB_1)$
and Lahoz-Macri-Stellari's work~\cite{LMS} describing
the image of point like objects in $D^b \Coh(\bB_0)$
under $\Theta$.

The next step is to describe the central 
charge $Z_G$ in terms of twisted sheaves on a K3 surface. 
By~\cite{Kuz}, if $X$
is a general cubic fourfold containing a plane, the
sheaf of non-commutative algebras $\bB_0$ 
is a push-forward of a sheaf of Azuyama algebras $\bB_S$ on 
a smooth K3 surface obtained as a double cover $S \to \mathbb{P}^2$.
The category of right $\bB_S$-modules is 
equivalent to 
$\Coh(S, \alpha)$, which is 
the category of coherent sheaves on $S$
twisted by an element in the Brauer group $\alpha \in \mathrm{Br}(S)$
with $\alpha^2 =1$. 
This provides another equivalence
\begin{align*}
\Upsilon \colon D^b \Coh(S, \alpha) \stackrel{\sim}{\to} D^b \Coh(\bB_0). 
\end{align*}
As a summary, there is a sequence of equivalences
\begin{align*}
D^b \Coh(S, \alpha) \stackrel{\Upsilon}{\to}
D^b \Coh(\bB_0) \stackrel{\Theta}{\to} \dD_X 
\stackrel{\Phi_1}{\leftarrow} \HMF(W). 
\end{align*}
We compute the pull-back of 
the central charge $Z_G$ on $\HMF(W)$
by the above sequence of equivalences, using 
the result of Proposition~\ref{intro:shift}. 
The resulting central charge on $D^b \Coh(S, \alpha)$
coincides with an integral over $S$
which appeared in Bridgeland's paper~\cite{Brs2}:   
\begin{prop}\emph{(Proposition~\ref{thm:int}.)}\label{intro:propZG}
There is an element $\mathfrak{B} \in H^{1, 1}(S, \mathbb{Q})$
and $c\in \mathbb{C}^{\ast}$ such that we have 
\begin{align*}
Z_G \circ \Phi_1^{-1} \circ \Theta \circ \Upsilon(E)
=c \cdot \int_S e^{\mathfrak{B}-\frac{\sqrt{-3}}{4}h} \ch(E) \sqrt{\td_S}
\end{align*}
for any $E\in D^b \Coh(S, \alpha)$. 
Here $h$ is a hyperplane in $\mathbb{P}^2$ pulled back to $S$. 
\end{prop}
The Chern character on $D^b (S, \alpha)$ is 
the \textit{untwisted} Chern character, 
defined to be the twisted Chern character by Huybrechts-Stellari~\cite{HuSt}, 
multiplied by the exponential of the minus of the B-field 
(cf.~Subsection~\ref{subsec:twisted}) to get back to the untwisted one.
Although it takes values in algebraic classes, it is no longer 
defined in the integer coefficient. 

The final step is to construct a corresponding 
Gepner type stability condition on $D^b \Coh(S, \alpha)$, using the above descriptions of 
the grade shift functor and the central charge. 
In this step, we need a further genericity assumption: 
the Brauer class $\alpha$ is non-trivial. 
This condition is not satisfied only if $X$ lies in
a union of countable many hypersurfaces in the 
moduli space of cubic fourfolds containing a plane. 
The following result obviously implies 
Theorem~\ref{intro:mainthm} as desired: 
\begin{thm}\emph{(Theorem~\ref{thm:final}.)}\label{thm:introGepner}
Suppose that $\alpha \neq 1$.
Then there is a Gepner type stability condition 
on $D^b \Coh(S, \alpha)$
with respect to $(\Upsilon^{-1} \circ F_B \circ \Upsilon, 
2/3)$, whose central charge is 
given by $Z_G \circ \Phi_1^{-1} \circ \Theta \circ \Upsilon$. 
\end{thm}
The $\alpha \neq 1$ condition puts a strong 
constraint on the image of $Z_G$, due to the 
non-existence of twisted line bundles, 
which makes it possible to prove Theorem~\ref{thm:introGepner}.  

The outline of the paper is as follows: 
in Section~\ref{sec:Pre}, we review some
background on stability conditions, 
graded matrix factorizations, Orlov/Kuznetsov 
equivalence, etc. 
In Section~\ref{sec:6}, we prove Proposition~\ref{intro:shift}. 
In Section~\ref{sec:const}, we prove Proposition~\ref{intro:propZG} and 
Theorem~\ref{thm:introGepner}. 
In Appendix~A, we review Chern characters
on graded matrix factorizations by Polishchuk-Vaintrob~\cite{PoVa}, 
and shows that the central charge $Z_G$ is numerical. 
In Appendix~B, we 
use the work by Bernardara-Macri-Mehrotra-Stellari~\cite{BMMS}
to
prove Theorem~\ref{thm:introGepner}.  
\subsection{Acknowledgment}
The author would like to thank 
Alexander Kuznetsov for answering the author's 
question, and Richard Thomas for 
a comment on Remark~\ref{rmk:thomas}. 
This work is supported by World Premier 
International Research Center Initiative
(WPI initiative), MEXT, Japan. This work is also supported by Grant-in Aid
for Scientific Research grant (22684002)
from the Ministry of Education, Culture,
Sports, Science and Technology, Japan.

\section{Preliminary}\label{sec:Pre}
This section is devoted to giving some preliminary
background 
required in this paper. 
\subsection{Numerical Bridgeland stability condition}
Let $\dD$ be a $\mathbb{C}$-linear
triangulated category, satisfying 
\begin{align}\label{Homfin}
\sum_{i\in \mathbb{Z}}
\dim \Hom(E, F[i]) <\infty
\end{align}
for any $E, F \in \dD$.
Let $K(\dD)$ be the Grothendieck group of $\dD$. 
There is 
an Euler bilinear pairing $\chi$
on $K(\dD)$, defined by 
\begin{align*}
\chi(E, F) \cneq 
\sum_{i} (-1)^i \dim \Hom(E, F[i]). 
\end{align*}
The \textit{numerical Grothendieck group} of $\dD$ is defined to be 
\begin{align*}
N(\dD) \cneq K(\dD)/\equiv
\end{align*}
where $E \equiv E'$ if 
$\chi(E, F)=\chi(E', F)$
for any $F \in \dD$. 
In what follows, we always assume that $N(\dD)$ 
is finitely generated. 
This condition is satisfied 
when $\dD=D^b \Coh(X)$ 
for a smooth projective variety $X$, and 
$N(\dD)$ is denoted by $N(X)$ in this case. 
Let us recall numerical Bridgeland
stability conditions on $\dD$. 
\begin{defi}\label{defi:stab} \emph{(\cite{Brs1})}
A numerical stability condition 
on $\dD$ consists of data
\begin{align}\label{pair2}
Z \colon N(\dD) \to \mathbb{C}, \quad 
\aA \subset \dD
\end{align}
where $Z$
is a group homomorphism 
(called central charge) 
and $\aA$ is the heart of bounded t-structure on $\dD$, 
satisfying the following conditions: 
\begin{itemize}
\item For any $0\neq E \in \aA$, we have 
\begin{align}\label{cond:1}
Z(E) \in  \mathbb{H} \cneq \{ r\exp(\sqrt{-1} \pi \phi) :
r>0, 0< \phi \le 1\}. 
\end{align}
\item Any object $E \in \aA$ admits a filtration 
(called Harder-Narasimhan filtration)
\begin{align*}
0=E_0 \subset E_1 \subset \cdots \subset E_N=E
\end{align*}
such that each subquotient $F_i=E_i/E_{i-1}$ is 
$Z$-semistable satisfying  
$\arg Z(F_i)> \arg Z(F_{i+1})$ for all $i$. 
\end{itemize}
Here an object $F \in \aA$ is $Z$-(semi)stable if 
for any non-zero $F' \subset F$, we have the 
inequality 
$\arg Z(F') \le(<) \arg Z(F)$
in $(0, \pi]$. 
\end{defi}
\begin{rmk}
By the construction, the Euler pairing $\chi$
descends to the perfect pairing
\begin{align*}
\chi \colon N(\dD) \times N(\dD) \to \mathbb{Z}. 
\end{align*}
Therefore the central charge $Z$
is always written as $\chi(u, -)$ for some 
$u\in \mathbb{C}$. 
\end{rmk}
Another way to give a numerical 
stability condition is 
to giving data
\begin{align}\label{num:pair}
(Z, \{\pP(\phi)\}_{\phi \in \mathbb{R}})
\end{align}
where $Z \colon N(\dD) \to \mathbb{C}$ is a group 
homomorphism, 
$\pP(\phi) \subset \dD$ are full 
subcategories (called \textit{semistable objects with phase} $\phi$)
 satisfying 
some axiom~\cite[Definition~1.1]{Brs1}. 
Given data (\ref{pair2}), 
the subcategories $\pP(\phi)$ for $0<\phi \le 1$
are defined to be $Z$-semistable objects
$E \in \aA$ such that the argument of $Z(E)$
coincides with $\pi \phi$. Other $\pP(\phi)$
are defined by the rule $\pP(\phi+1)=\pP(\phi)[1]$. 
Conversely given data (\ref{num:pair}), 
the heart $\aA$ is given by the 
extension closure of $\pP(\phi)$ for $0<\phi \le 1$. 
For the detail, see~\cite[Proposition~5.3]{Brs1}. 
The space of numerical stability conditions is defined as follows: 
\begin{defi}
The set 
$\Stab(\dD)$ is defined to be the set of numerical 
stability conditions (\ref{pair2}) on $\dD$, satisfying the 
\textit{support property}: 
\begin{align*}
\mathrm{sup} \left\{ \frac{\lVert E \rVert}{\lvert Z(E) \rvert} 
: 0\neq E\in \aA \mbox{ is } Z \mbox{-semistable} \right\} < \infty. 
\end{align*}
Here $\lVert \ast \rVert$ is a fixed norm on $N(\dD)_{\mathbb{R}}$. 
\end{defi}
In~\cite{Brs1}
(also see~\cite{K-S}), Bridgeland 
shows that there is a natural topology on 
$\Stab(\dD)$ such that 
the forgetting map 
\begin{align}\label{Z:forget}
\zZ \colon \Stab(\dD) \to N(\dD)_{\mathbb{C}}^{\vee}
\end{align}
sending a stability condition to its central charge
is a local homeomorphism. 
In particular, $\Stab(\dD)$ is a complex manifold. 

Let $\Aut(\dD)$
be the group of autequivalence on $\dD$. 
There is a left $\Aut(\dD)$-action on $\Stab(\dD)$.
For $\Phi \in \Aut(\dD)$, it acts on (\ref{num:pair}) 
as 
\begin{align*}
\Phi_{\ast} (Z, \{\pP(\phi)\}_{\phi \in \mathbb{R}})
=(Z \circ \Phi^{-1}, \{ \Phi(\pP(\phi)) \}_{\phi \in \mathbb{R}}).
\end{align*}
There is also a right $\mathbb{C}$-action on 
$\Stab(\dD)$. For $\lambda \in \mathbb{C}$, it acts on (\ref{num:pair})
as 
\begin{align*}
 (Z, \{\pP(\phi)\}_{\phi \in \mathbb{R}}) \cdot (\lambda)
= (e^{-\sqrt{-1}\pi \lambda} Z, \{ \pP(\phi + \Ree \lambda) \}_{\phi \in \mathbb{R}}).
\end{align*}
The notion of Gepner type stability conditions is 
defined as follows: 
\begin{defi}\emph{(\cite{TGep})}
A numerical stability condition $\sigma \in \Stab(\dD)$ is called 
Gepner type with respect to $(\Phi, \lambda) \in \Aut(\dD) \times \mathbb{C}$
if the following condition holds: 
\begin{align*}
\Phi_{\ast}\sigma= \sigma \cdot (\lambda). 
\end{align*}
\end{defi}

\subsection{Gepner type stability conditions on graded matrix factorizations}
Let $W$ be a homogeneous element 
\begin{align}\label{def:A}
W \in A \cneq \mathbb{C}[x_1, x_2, \cdots, x_n]
\end{align}
of degree $d$
such that
$(W=0) \subset \mathbb{C}^n$ has an isolated 
singularity at the origin. 
For a graded $A$-module $P$, 
we denote by $P_i$ its degree $i$-part, 
and $P(k)$ the graded $A$-module 
whose grade is shifted by $k$, i.e. 
$P(k)_{i}= P_{i+k}$. 
\begin{defi}
A graded matrix factorization of $W$ is data
\begin{align}\label{MF}
P^0 \stackrel{p^0}{\to} P^1 \stackrel{p^1}{\to}
P^0(d)
\end{align}
where $P^i$ are graded free $A$-modules of finite rank, $p^i$ 
are homomorphisms of graded $A$-modules, satisfying the 
following conditions:
\begin{align*}
\quad p^1 \circ p^0= \cdot W, \quad 
p^0(d) \circ p^1= \cdot W.
\end{align*}
\end{defi}
The category $\HMF(W)$ 
is defined to be the homotopy 
category of the dg-category of 
graded matrix factorizations of $W$ 
(cf.~\cite{Orsin}). 
 The grade shift functor
$P^{\bullet} \mapsto P^{\bullet}(1)$ 
induces the 
autequivalence $\tau$ of $\HMF(W)$, 
which satisfies the 
 following identity: 
\begin{align}\label{taud}
\tau^{\times d} =[2]. 
\end{align}
\begin{rmk}
There is a Serre functor on $\HMF(W)$ 
given by (cf.~\cite[Theorem~3.8]{KST2})
\begin{align}\label{Serre}
\sS_W =\tau^{d-n}[n-2].
\end{align}
In particular, $\sS_W^{\times d}=[n(d-2)]$, and 
$\HMF(W)$ is interpreted as a fractional 
Calabi-Yau category with dimension $n(1-2/d)$. 
This fact will be used in Appendix~B. 
\end{rmk}

Since $(W=0) \subset \mathbb{C}^n$ has an isolated 
singularity at the origin, the 
triangulated category $\HMF(W)$ satisfies
the condition (\ref{Homfin}), 
and it is finitely generated. 
(For instance, use Orlov's result in Theorem~\ref{thm:Orlov}
below.)
The following is the numerical version
of~\cite[Conjecture~1.2]{TGep}:
\begin{conj}\label{conj:main}
There is a Gepner type stability condition
\begin{align*}
\sigma_G=(Z_G, \{\pP_G(\phi)\}_{\phi \in \mathbb{R}})
\in  \Stab(\HMF(W))
\end{align*}
with respect to $(\tau, 2/d)$, whose
central charge $Z_G$ is given by 
\begin{align}\label{Z_G}
Z_G(P^{\bullet})= \mathrm{str}(e^{2\pi \sqrt{-1}/d} \colon 
P^{\bullet} \to P^{\bullet}). 
\end{align}
\end{conj}
The $e^{2\pi \sqrt{-1}/d}$-action on $P^{\bullet}=P^0 \oplus P^1$ is induced
by the $\mathbb{Z}$-grading on each $P^i$, and 
`str' is the supertrace which respects the $\mathbb{Z}/2\mathbb{Z}$-grading 
on $P^{\bullet}$.
The definition of the
central charge $Z_G$ 
first appeared in~\cite{Wal}. 
It
is more precisely 
written as follows: 
since $P^i$ are free $A$-modules of finite rank, 
they are written as 
\begin{align*}
P^i \cong 
\bigoplus_{j=1}^{m} A(n_{j}^{i}), 
\quad n_{j}^{i} \in \mathbb{Z}.
\end{align*}
Then (\ref{Z_G}) is written as 
\begin{align*}
Z_G(P^{\bullet})=
\sum_{j=1}^{m} \left( e^{2 n_{j}^{0} \pi \sqrt{-1}/d} - e^{2 n_{j}^{1} 
\pi \sqrt{-1} /d} \right). 
\end{align*}
\begin{rmk}\label{rmk:low}
In the low degree cases of Conjecture~\ref{conj:main},
there is nothing to prove in 
the $d=1$ case as $\HMF(W)=\{0\}$ in this case. 
In the $d=2$ case, the Kn{\"o}rrer periodicity~\cite{Knor}
allows us to reduce to 
the case of $n=1$ or $n=2$, and Conjecture~\ref{conj:main}
in these cases are checked in~\cite{Tak}, \cite{TGep}. 
\end{rmk}
The following lemma, 
which is an obvious necessary condition for 
Conjecture~\ref{conj:main}, is an easy consequence of 
an interpretation of $Z_G$ in terms of 
a Chern character of graded matrix factorizations, and
Hirzebruch-Grothendieck Riemann-Roch formula~\cite{PoVa}. 
The detail will be provided in Appendix~A. 
\begin{lem}\label{lem:later}
The central charge $Z_G$ factors through the 
canonical surjection 
$K(\HMF(W)) \twoheadrightarrow N(W)$. 
In particular, it is written as
\begin{align*}
Z_G(P^{\bullet})= \chi(u, P^{\bullet})
\end{align*}
for some $u\in N(W)_{\mathbb{C}}$
with $\tau_{\ast}^{-1} u= e^{2\pi \sqrt{-1}/d}u$. 
\end{lem}

We now recall Orlov's theorem~\cite{Orsin} which 
relates $\HMF(W)$ with 
the derived category of coherent sheaves on the hypersurface
\begin{align*}
X \cneq (W=0) \subset \mathbb{P}^{n-1}
\end{align*}
by semiorthogonal decompositions (SOD for short). 
Since we only use the case of $n>d$, we 
give a statement in this case. 
\begin{thm} \emph{(\cite[Theorem~2.5]{Orsin})}\label{thm:Orlov}
If $n>d$, then  
there is a fully 
faithful embedding for each $i\in \mathbb{Z}$
\begin{align*}
\Phi_i \colon \HMF(W) \hookrightarrow D^b \Coh(X)
\end{align*}
and SOD
\begin{align}\label{Or:SOD}
D^b \Coh(X)= \langle \oO_X(-i-n+d+1), 
\cdots, \oO_X(-i) , \Phi_i \HMF(W) 
\rangle. 
\end{align}
\end{thm}

\subsection{Cubic fourfolds containing a plane and K3 surfaces}
\label{subsec:cubic}
Let $X$
be a cubic fourfold which contains a plane $P$
\begin{align*}
\mathbb{P}^2 =P \subset X =(W=0) \stackrel{i}{\hookrightarrow} \mathbb{P}^5. 
\end{align*}
We recall a relationship between such cubic fourfolds and K3 surfaces 
obtained as double covers of $\mathbb{P}^2$. 
Let
\begin{align*}
p \colon \widetilde{\mathbb{P}}^{5} \to \mathbb{P}^5, \ 
\sigma \colon \widetilde{X} \to X
\end{align*}
be the blow-ups at the plane $P \subset \mathbb{P}^5$, 
$P \subset X$ respectively. 
The exceptional divisors of $p$, $\sigma$ are denoted by 
$D' \subset \widetilde{\mathbb{P}}^5$, $D\subset \widetilde{X}$ respectively. 
The linear projection from $P$ gives morphisms
\begin{align*}
q \colon \widetilde{\mathbb{P}}^5 \to \mathbb{P}^2, \ 
\pi \cneq q \circ j \colon \widetilde{X} \to \mathbb{P}^2
\end{align*}
where $j$ is the inclusion 
$\widetilde{X} \hookrightarrow \widetilde{\mathbb{P}}^5$. 
The morphism $q$ is 
the projectivization of the following rank four vector 
bundle $E$ on $\mathbb{P}^2$. 
\begin{align}\label{basis}
E=\oO_{\mathbb{P}^2} {\bf e_1} \oplus \oO_{\mathbb{P}^2} {\bf e_2}
\oplus 
\oO_{\mathbb{P}^2} {\bf e_3} \oplus \oO_{\mathbb{P}^2}(-1) {\bf f}.
\end{align}
We will usually abbreviate the basis elements 
${\bf e_1, e_2, e_3, f}$ unless it is necessary to specify them. 
Below we denote by $h, H$
(resp.~$h', H'$) the classes of hyperplanes in $\mathbb{P}^2$, $\mathbb{P}^5$ 
pulled back to 
$\widetilde{X}$ (resp.~$\widetilde{\mathbb{P}}^5$) respectively. 
Note that we have the following relations:
\begin{align}\label{easy}
D=H-h, \ D'=H'-h'. 
\end{align}
Hence we have 
\begin{align*}
\widetilde{X} \in \lvert 3H'-D' \rvert =
\lvert 2H' +h' \rvert.
\end{align*}
In particular, $\pi$ is a quadric fibration, and the defining 
equation of $\widetilde{X}$ gives a morphism
\begin{align}\label{mor:s}
s \colon \oO_{\mathbb{P}^2}(-1) \to \mathrm{Sym}^2 E^{\vee}. 
\end{align} 
The morphism $s$ induces the morphism
\begin{align}\label{s'}
s' \colon E \to E^{\vee}(1). 
\end{align}
The morphism $\pi$ has degenerated fibers along the 
zero locus of 
\begin{align}\notag
\det(s') \in \Hom(\det E, \det (E^{\vee}(1))) \cong 
H^0(\mathbb{P}^2, \oO_{\mathbb{P}^2}(6))
\end{align}
which is a sextic $C \subset \mathbb{P}^2$. 
Let 
\begin{align}\label{K3:double}
f \colon S \to \mathbb{P}^2
\end{align}
be the double cover branched along $C$. 
The curve $C$ is non-singular for a general cubic 
fourfold containing a plane. In this case, the associated 
double cover $S$ is a smooth projective K3 surface. 
In what follows, we assume that the cubic fourfold $X$ is 
general so that $C$ is non-singular. 
The covering involution of $f$ is denoted by $\iota$, 
and (by abuse of notation) we denote by $h$ 
the class of a hyperplane in $\mathbb{P}^2$
pulled back to $S$. 
The relevant diagram in this subsection
is summarized below:
\begin{align*}
\xymatrix{
(P \subset X) \ar[r]^{i}  
 & (P \subset \mathbb{P}^5) & \\
(D \subset \widetilde{X}) \ar[r]^{j} \ar[u]^{\sigma}
 \ar[rd]_{\pi} & (D' \subset \widetilde{\mathbb{P}}^5) \ar[d]^{q} \ar[u]_{p} &
S\ar[d]^{\iota} \ar[ld]_{f} \\
  & (\mathbb{P}^2 \supset C) & \ar[l]_{f} S
}
\end{align*}

\subsection{Sheaves of Clifford algebras and twisted K3 surfaces}\label{subsec:sheaves}
Similarly to the classical construction of Clifford algebras, 
the morphism (\ref{mor:s}) defines the sheaf of Clifford algebras $\bB_{s}$
on $\mathbb{P}^2$ (cf.~\cite[Section~3]{Kuz}). 
It has an even part $\bB_{0}$, which is described as
\begin{align}\label{basis*}
\bB_0 &= \oO_{\mathbb{P}^2} \oplus 
(\wedge^2 E \otimes \oO_{\mathbb{P}^2}(-1)) \oplus 
(\wedge^4 E \otimes \oO_{\mathbb{P}^2}(-2)) \\
\notag
&\cong \oO_{\mathbb{P}^2} \oplus \oO_{\mathbb{P}^2}(-1)^{\oplus 3}
\oplus \oO_{\mathbb{P}^2}(-2)^{\oplus 3} \oplus \oO_{\mathbb{P}^2}(-3).
\end{align}
It also has an odd part $\bB_1$, given by 
\begin{align}\label{basis**}
\bB_1 &= 
E \oplus 
(\wedge^3 E \otimes \oO_{\mathbb{P}^2}(-1))  \\
\notag
&\cong \oO_{\mathbb{P}^2}^{\oplus 3} \oplus 
\oO_{\mathbb{P}^2}(-1)^{\oplus 2} \oplus \oO_{\mathbb{P}^2}(-2)^{\oplus 3}. 
\end{align}
We also 
define $\bB_i$ for $i\in \mathbb{Z}$
by the rule $\bB_{i+2}=\bB_i(1)$. 
By~\cite[Corollary~3.9]{Kuz},
every sheaves $\bB_i$ are flat over $\bB_0$ and 
we have 
\begin{align}\label{Bkl}
\bB_i \otimes_{\bB_0} \bB_{j} \cong \bB_{i+j}, \ \mbox{ for all } i, j \in \mathbb{Z}. 
\end{align} 
In particular, for every $i$ there is an equivalence of abelian categories
\begin{align}\label{B0k}
\otimes_{\bB_0} \bB_i \colon \Coh(\bB_0) \stackrel{\sim}{\to}
\Coh(\bB_0). 
\end{align}
Here $\Coh(\bB_0)$ is the abelian category of coherent 
right $\bB_0$-modules on $\mathbb{P}^2$. 

Let $S$ be the K3 surface obtained as a double cover (\ref{K3:double}).
By~\cite[Section~3.5]{Kuz}, there exists 
a sheaf of Azumaya algebras $\bB_S$ on $S$ such that 
$f_{\ast} \bB_S = \bB_0$, and an equivalence
\begin{align}\label{Eq:K3}
f_{\ast} \colon \Coh(\bB_S) \stackrel{\sim}{\to} \Coh(\bB_0). 
\end{align} 
The abelian categories $\Coh(\bB_0)$, 
$\Coh(\bB_S)$ are also 
described in terms of 
twisted sheaves. 
There exists an element in the Brauer group
\begin{align*}
\alpha \in \mathrm{Br}(S)=H^2(S, \oO_S^{\ast}), \quad 
\alpha^2=\id
\end{align*}
and an
$\alpha$-twisted vector bundle $\uU_{0}$ 
of rank two such that $\bB_S= \eE nd(\uU_{0})$ 
and the functor
\begin{align*}
\Coh(S, \alpha) \ni F \mapsto
\uU_{0}^{\vee} \otimes F \in \Coh(\bB_S)
\end{align*}
is an equivalence. 
Here $\Coh(S, \alpha)$ is the abelian 
category of $\alpha$-twisted coherent sheaves on $S$
(cf.~\cite[Section~1]{HuSt}).
Combined with the above equivalences, we obtain 
the equivalence
\begin{align}\label{Upsilon}
\Upsilon(-) \cneq f_{\ast}(\uU_{0}^{\vee} \otimes -) \colon 
D^b \Coh(S, \alpha) \stackrel{\sim}{\to} D^b \Coh(\bB_0). 
\end{align}

\subsection{Orlov/Kuznetsov equivalence}
Let $\dD_X$ be the semiorthogonal summand of $D^b \Coh(X)$, 
defined by
\begin{align}\label{SOD:kuz}
D^b \Coh(X)= \langle \dD_X, \oO_X, \oO_X(1), \oO_X(2) \rangle. 
\end{align}
In~\cite{Kuz2}, 
Kuznetsov established an equivalence
between $\dD_X$ and $D^b \Coh(\bB_0)$. 
A starting point is the fully faithful functor
\begin{align*}
\Phi \colon D^b \Coh(\bB_0) \to D^b \Coh(\widetilde{X})
\end{align*}
constructed in~\cite{Kuz}, 
defined as a Fourier-Mukai transform 
\begin{align}\label{PhiFM}
\Phi(-)=\pi^{\ast}(-) \otimes_{\pi^{\ast}\bB_0} \eE.
\end{align} 
Here $\eE$ is a sheaf of left 
$\pi^{\ast}\bB_0$-modules on $\widetilde{X}$
constructed as follows: 
by~\cite[Corollary~3.12]{Kuz},
there are injections as left $q^{\ast}\bB_0$-modules
for each $i\in \mathbb{Z}$
\begin{align}\label{delta}
\delta_i \colon q^{\ast} \bB_i \to q^{\ast}\bB_{i+1}(H').
\end{align} 
By an abuse of notation, we will also 
denote by $\delta_i$ the twist of the 
above morphism by any line bundle. 
Then $j_{\ast}\eE$ is given by the cokernel
of the above morphism for $i=0$
\begin{align}\label{canonical}
0 \to q^{\ast}\bB_0(-2H') \stackrel{\delta_0}\to q^{\ast}\bB_1(-H') \to 
j_{\ast} \eE \to 0.
\end{align}
As $\oO_{\widetilde{X}}$-module, the sheaf $\eE$ is locally free of rank four.
 
Kuznetsov~\cite{Kuz2} performs a sequence of mutations
of SOD of $D^b \Coh(\widetilde{X})$, 
and replaces $\Phi$ by another fully faithful functor 
$\Phi''$ 
\begin{align*}
\Phi'' \cneq \dL_{\oO_{\widetilde{X}}(h-H)} \circ \dR_{\oO_{\widetilde{X}}(-h)}
\circ \Phi \colon D^b \Coh(\bB_0) \to D^b \Coh(\widetilde{X})
\end{align*}
where $\dL_{\oO_{\widetilde{X}}(h-H)}$ and 
$\dR_{\oO_{\widetilde{X}}(-h)}$
are defined to be
\begin{align*}
\dL_{\oO_{\widetilde{X}}(h-H)}(-) &\cneq \mathrm{Cone}
\left(\dR \Hom(\oO_{\widetilde{X}}(h-H), -) \otimes \oO_{\widetilde{X}}(h-H)
\to -  \right) \\
\dR_{\oO_{\widetilde{X}}(-h)}(-) &\cneq \mathrm{Cone}
\left(- \to 
\dR \Hom(-, \oO_{\widetilde{X}}(-h))^{\vee} \otimes \oO_{\widetilde{X}}(-h)
\right)[-1]. 
\end{align*}
Then it is shown that the image of $\Phi''$
coincides with the image of the pull-back 
of the blow-up $\sigma \colon \widetilde{X} \to X$
restricted to $\dD_X$.   
Applying $\dR \sigma_{\ast}$, 
the following result is obtained in~\cite{Kuz2}:
\begin{thm}\emph{(\cite{Kuz2})}\label{thm:Theta}
The functor 
\begin{align}\label{Theta}
\Theta \cneq \dR \sigma_{\ast} \circ \Phi'' \colon 
D^b \Coh(\bB_0) \stackrel{\sim}{\to} \dD_X
\end{align}
is an equivalence.
\end{thm}

It is useful to write
$\Theta(F)$ for $F \in D^b \Coh(\bB_0)$ as follows
(cf.~\cite[Theorem~4.3, Step~7]{Kuz2}):
\begin{align}\notag
\Theta(F) =
\{ \dR \Hom(\oO_{\widetilde{X}}(h-H), &\Phi(F)) \otimes I_P
\to \dR \sigma_{\ast} \Phi(F) \\
\label{explicit:Kuz}
&\to
 \dR \Hom(\Phi(F), \oO_{\widetilde{X}}(-h))^{\vee}
\otimes \oO_X(-1) \}. 
\end{align}
Here $I_P \subset \oO_X$ is the ideal sheaf of $P$, which 
is easily checked to be an object in $\dD_X$.  

Now we combine $\Theta$ with Orlov equivalence. 
Note that, since $\omega_X =\oO_X(-3)$, 
the SOD (\ref{SOD:kuz}) induces another SOD
\begin{align*}
D^b \Coh(X)=\langle \oO_X(-3), \oO_X(-2), \oO_X(-1), \dD_X \rangle. 
\end{align*}
Therefore Theorem~\ref{thm:Orlov} yields an equivalence
\begin{align*}
\Phi_1 \colon \HMF(W) \stackrel{\sim}{\to} \dD_X. 
\end{align*}
We summarize the equivalences obtained so far 
in the following corollary: 
\begin{cor}\label{cor:sum}
There is a sequence of equivalences
\begin{align*}
D^b \Coh(S, \alpha) \stackrel{\Upsilon}{\to}
D^b \Coh(\bB_0) \stackrel{\Theta}{\to} \dD_X 
\stackrel{\Phi_1}{\leftarrow} \HMF(W). 
\end{align*}
Here $\Upsilon$ is given in (\ref{Upsilon}), 
$\Theta$ is given in Theorem~\ref{thm:Theta}
and $\Phi_1$ is given in Theorem~\ref{thm:Orlov}. 
\end{cor}

\section{Description of the grade shift functor}\label{sec:6}
The purpose of this section is to prove Proposition~\ref{intro:shift}. 
In what follows, we always assume that $X$ is a cubic 
fourfold containing a plane $P$, which is general 
so that the associated K3 surface $S$ is smooth
(cf.~Subsection~\ref{subsec:cubic}).
\subsection{Summary of the result}
Let us consider the equivalence in Corollary~\ref{cor:sum}:
\begin{align*}
\Phi_1^{-1} \circ \Theta \colon D^b \Coh(\bB_0) \stackrel{\sim}{\to}
\HMF(W). 
\end{align*}
We are going 
to describe the grade shift functor $\tau$
on $\HMF(W)$ in terms of $D^b \Coh(\bB_0)$
under the above equivalence. 
We first recall the description of $\tau$ in 
terms of the autequivalence in $\dD_X$, given in~\cite{BFK}.  
Let us consider the functor
\begin{align*}
F_X \colon 
D^b \Coh(X) \to D^b \Coh(X). 
\end{align*}
defined to be
\begin{align}\label{def:FX}
F_X(-) \cneq \mathrm{Cone}\left( \dR \Hom(\oO_X, -\otimes \oO_X(1)) \otimes \oO_X \to -\otimes \oO_X(1)
  \right). 
\end{align}
The functor $F_X$ preserves
$\dD_X$, and by~\cite[Lemma~1.10]{LMS} it gives an autequivalence of 
$\dD_X$. 
\begin{prop}\emph{(\cite[Proposition~5.8]{BFK})}\label{prop:BFK}
The following diagram commutes: 
\begin{align*}
\xymatrix{
\HMF(W) \ar[r]^{\quad \Phi_1} \ar[d]_{\tau} & \dD_X \ar[d]^{F_X} \\
\HMF(W) \ar[r]^{\quad \Phi_1} & \dD_X.
}
\end{align*}
\end{prop}
By the above result, it is enough to describe the 
functor $F_X$ in terms of $D^b \Coh(\bB_0)$. 
We observe the following:
\begin{lem}\label{lem:spherical}
For all $i\in \mathbb{Z}$, we have 
\begin{align*}
&\dR \Hom_{\bB_0}(\bB_i, \bB_i) \cong \mathbb{C} \oplus \mathbb{C}[-2] \\
&\dR \Hom_{\bB_0}(\bB_i, \bB_{i+1}) \cong \mathbb{C}^{3} \\
&\dR \Hom_{\bB_0}(\bB_i, \bB_{i+2}) \cong \mathbb{C}^{6}. 
\end{align*}
\end{lem}
\begin{proof}
By the equivalence (\ref{B0k}), we may assume that 
$i=0$. Then the result easily follows from 
\begin{align*}
\dR\Hom_{\bB_0}(\bB_0, \bB_k) \cong 
\dR\Hom_{\mathbb{P}^2}(\oO_{\mathbb{P}^2}, \bB_k).
\end{align*}
\end{proof}
As noted in~\cite[Remark~2.1]{StMa},
the above lemma shows that  
the objects $\bB_i$
are spherical objects~\cite{ST}. 
The associated spherical twists
and their inverses are given by 
\begin{align*}
&\mathrm{ST}_{\bB_i}(-) \cneq \mathrm{Cone} \left( \dR \Hom(\bB_i, -)
\otimes \bB_i \to - \right) \\
&\mathrm{ST}_{\bB_i}^{-1}(-) \cneq \mathrm{Cone} \left( 
- \to \dR \Hom(-, \bB_i)^{\vee} \otimes \bB_i \right)[-1]. 
\end{align*}
The above functors are autequivalences of $D^b \Coh(\bB_0)$. 
Combined with (\ref{B0k}), we 
define the following autequivalence 
\begin{align}\label{eq:FB}
F_B \cneq \mathrm{ST}_{\bB_1}^{-1} \circ \otimes_{\bB_0} \bB_{-1}[1]. 
\end{align}
The following proposition is the main result in this section:
\begin{prop}\label{key:prop}
The following diagram commutes: 
\begin{align*}
\xymatrix{
D^b \Coh(\bB_0) \ar[r]^{\quad \Theta} \ar[d]_{F_B} & \dD_X \ar[d]^{F_X} \\
D^b \Coh(\bB_0) \ar[r]^{\quad \Theta} & \dD_X.
}
\end{align*}
\end{prop}
Combined with Proposition~\ref{prop:BFK} and 
(\ref{taud}), we obtain the following corollary:
\begin{cor}\label{cor:shift}
The following diagram commutes:
\begin{align*}
\xymatrix{
D^b \Coh(\bB_0) \ar[r]^{\Phi_1^{-1} \circ 
\Theta} \ar[d]_{F_B} & \HMF(W) \ar[d]^{\tau} \\
D^b \Coh(\bB_0) \ar[r]^{\Phi_1^{-1} \circ \Theta} & \HMF(W).
}
\end{align*}
In particular $F_B^{\times 3}$ is isomorphic to $[2]$. 
\end{cor}
A proof of 
Proposition~\ref{key:prop}
will be given in Subsection~\ref{subsec:proof}.

\subsection{Explicit description of $\delta_i$}\label{subsec:explicit}
The purpose of this subsection is to give an explicit 
description of $\delta_i$ in (\ref{delta}) 
in terms of the cubic polynomial $W$, which 
will be
 relevant in some computations required in the proof of 
Proposition~\ref{key:prop}. 
Let 
\begin{align*}
[x_1 \colon x_2 \colon x_3 \colon x_4 \colon x_5 \colon x_6]
\end{align*}
be the homogeneous coordinate of $\mathbb{P}^5$.  
Without loss of generality, we may assume that 
\begin{align*}
P=\{ x_1=x_2=x_3=0\} \subset \mathbb{P}^5. 
\end{align*}
Since $X$ contains $P$, the homogeneous polynomial $W$ is written as
\begin{align*}
W=W'(x_1, x_2, x_3)+ \sum_{1\le i, j\le 3}x_i x_j
W_{ij}(x_4, x_5, x_6) + \sum_{1\le i\le 3} x_i W_i(x_4, x_5, x_6)
\end{align*}
such that
\begin{itemize}
\item $W'(x_1, x_2, x_3)$ is a homogeneous cubic polynomial 
in $x_1, x_2, x_3$.
\item $W_{ij}(x_4, x_5, x_6)$ is a 
linear combination of $x_4, x_5, x_6$.
\item 
$W_i(x_4, x_5, x_6)$ is a homogeneous quadric polynomial in $x_4, x_5, x_6$.
\end{itemize} 
Let us take an 
element
\begin{align*}
x_7 \in H^0(\wP^5, \oO_{\wP^5}(D'))
\end{align*}
which corresponds to $1$ under the natural 
isomorphism
\begin{align*}
H^0(\mathbb{P}^5, \oO_{\mathbb{P}^5}) \stackrel{\cong}{\to}
H^0(\wP^5, \oO_{\wP^5}(D')). 
\end{align*}
Since $p$ is a blow-up at $P$, 
for $1\le i\le 3$ we have 
\begin{align}\label{xrelation}
x_i=x_7 y_i \ \mbox{ for some } \
y_i \in H^0(\wP^5, \oO_{\wP^5}(H'-D')). 
\end{align}
 Here by abuse of notation, 
the pull-back of $x_i \in H^0(\mathbb{P}^5, \oO_{\mathbb{P}^5}(1))$
to $\wP^5$ is denoted by the same symbol $x_i$. 
We define the following polynomial
\begin{align*}
\widetilde{W} \cneq 
W'(y_1, &y_2, y_3)x_7^2+ \\
&\sum_{1\le i, j\le 3}y_i y_j x_7
W_{ij}(x_4, x_5, x_6) + \sum_{1\le i\le 3} y_i W_i(x_4, x_5, x_6). 
\end{align*}
By (\ref{easy}), the above polynomial 
makes sense as
\begin{align*}
\widetilde{W} \in 
H^0(\wP^5, \oO_{\wP^5}(2H'+h')).
\end{align*}
The above polynomial $\widetilde{W}$ is the defining 
equation of $\wX$ in $\wP^5$, and 
the morphism $s'$ in (\ref{s'}) is given by 
the Hessian of $\widetilde{W}$, i.e. 
by regarding local sections of $E$, $E^{\vee}$ as
column vectors 
with respect to the basis (\ref{basis})
and its dual basis ${\bf e_1^{\ast}, e_2^{\ast}, e_3^{\ast}, f^{\ast}}$, 
and setting $\partial_i \cneq \partial/\partial x_i$, 
the morphism $s'$ is written as a matrix
\begin{align}\label{m1}
s'=\frac{1}{2}\left( \begin{array}{cccc}
\partial_4 \partial_4 \widetilde{W} & \partial_4 \partial_5 \widetilde{W} & 
\partial_4 \partial_6 \widetilde{W} & \partial_4 \partial_7 \widetilde{W} \\
\partial_5 \partial_4 \widetilde{W} & \partial_5 \partial_5 \widetilde{W} & 
\partial_5 \partial_6 \widetilde{W} & \partial_5 \partial_7 \widetilde{W} \\
\partial_6 \partial_4 \widetilde{W} & \partial_6 \partial_5 \widetilde{W} & 
\partial_6 \partial_6 \widetilde{W} & \partial_6 \partial_7 \widetilde{W} \\
\partial_7 \partial_4 \widetilde{W} & \partial_7 \partial_5 \widetilde{W} & 
\partial_7 \partial_6 \widetilde{W} & \partial_7 \partial_7 \widetilde{W} 
\end{array} \right).  
\end{align}
Here we regard $y_i$ as an element in 
$H^0(\mathbb{P}^2, \oO_{\mathbb{P}^2}(1))$ by the relation (\ref{easy}). 

Let us consider the morphism $\delta_i$ in (\ref{delta}). 
It is easy to see that $\delta_i$ is adjoint to the 
morphism
\begin{align*}
\delta_i' \colon 
\bB_{i} \to q_{\ast} (q^{\ast} \bB_{i+1}(H')) \cong 
\bB_{i+1} \otimes E^{\vee}
\end{align*}
induced by the morphism
$\bB_{i} \otimes E \to \bB_{i+1}$
defined by the Clifford multiplication:  
\begin{align*}
(\wedge, \lrcorner \circ (\id \otimes s')) \colon 
\wedge^j E \otimes E
\to \wedge^{j+1} E \oplus \wedge^{j-1} E(H'). 
\end{align*}
Here $\wedge$ is taking the right wedge product, 
$s'$ is the morphism (\ref{s'})
and 
$\lrcorner$ is the right contraction.  
The morphism (\ref{delta}) is obtained by the composition
\begin{align}\label{comp}
q^{\ast} \bB_i \stackrel{q^{\ast}\delta_i'}{\to}
q^{\ast}\bB_{i+1} \otimes q^{\ast}E^{\vee}  \to 
q^{\ast}\bB_{i+1}(H')
\end{align}
where the right morphism is induced by the tautological 
surjection
$q^{\ast}E^{\vee} \twoheadrightarrow \oO_{\wP^5}(H')$, 
which is the right contraction by the element
\begin{align}\label{wedge*}
x_4 {\bf e_1} + x_5 {\bf e_2} + x_6 {\bf e_3} + x_7 {\bf f}. 
\end{align}
We also have the following: 
\begin{align*}
(x_4, x_5, x_6, x_7) \circ
s'
=\frac{1}{2} \left( \partial_4 \widetilde{W}, \partial_5 \widetilde{W}, 
\partial_6 \widetilde{W}, \partial_7 \widetilde{W} \right).
\end{align*}
Hence the composition (\ref{comp})
is the sum of the right wedge product by (\ref{wedge*})
and the right contraction by 
\begin{align}\label{contract*}
\frac{1}{2} \left(\partial_4 \widetilde{W} {\bf e_1^{\ast}}
+ \partial_5 \widetilde{W} {\bf e_2^{\ast}} + 
\partial_6 \widetilde{W} {\bf e_3^{\ast}} + 
\partial_7 \widetilde{W} {\bf f^{\ast}}\right).
\end{align}

\subsection{Some cohomology computations}\label{subsec:compute}
This subsection is devoted to do
some cohomology computations, which will 
be used in the next subsection. 
\begin{lem}\label{Mkl}
We set $M_{k, l}$ to be
\begin{align*}
M_{k, l} \cneq \dR \Gamma(\widetilde{\mathbb{P}}^5, 
\oO_{\widetilde{\mathbb{P}}^5}(kH' + lh') ). 
\end{align*}
Then we have 

(i) $M_{k, l}=0$ if
$-3 \le k\le -1$, and $M_{0, 0}=\mathbb{C}$, 
$M_{0, -3}=\mathbb{C}[-2]$. 

(ii)
$M_{k, l}=0$ if 
$-2 \le l\le 0$ with $-6<k+l<0$. 
Moreover we have
$M_{l, -6-l}=\mathbb{C}[-5]$ for $-2 \le l\le 0$. 

(iii) 
$M_{-4, -3}=\mathbb{C}^{3}[-5]$.
\end{lem}
\begin{proof}
If $k\ge -3$,
we can compute $M_{k, l}$ from  
\begin{align}\label{com(i)}
M_{k, l}=
\dR \Gamma(\mathbb{P}^2, \mathrm{Sym}^{k} E^{\vee} (l)). 
\end{align}
Here we set $\mathrm{Sym}^k E^{\vee}=0$ for $k<0$. 
Hence we immediately obtain (i).
We can also describe $M_{k, l}$ as 
\begin{align*}
M_{k, l} &= \dR \Gamma(\widetilde{\mathbb{P}}^5, 
\oO_{\widetilde{\mathbb{P}}^5}((k+l)H'-l D') ) \\
& = \dR \Gamma(\mathbb{P}^5, \oO_{\mathbb{P}^5}(k+l) \otimes \dR p_{\ast} 
\oO_{\wP^5}(-lD')). 
\end{align*}
Hence (ii) follows 
using 
\begin{align}\label{Rp}
\dR p_{\ast} \oO_{\widetilde{\mathbb{P}}^5}(-lD')=\oO_{\mathbb{P}^5}, \ 
-2 \le l \le 0.  
\end{align}
Finally since $K_{\wP^5}=-4H'-2h'$, the Serre duality 
implies
\begin{align*}
M_{k, l}=M_{-k-4, -l-2}^{\vee}[-5]. 
\end{align*}
Hence we have $M_{-4, -3}=M_{0, 1}^{\vee}[-5]$, 
which coincides with $\mathbb{C}^{3}[-5]$
by (\ref{com(i)}). 
\end{proof}
We will also use the following computations:
\begin{lem}\label{lem:com2}
We have 
\begin{align}\label{com0}
&\dR \Hom(\oO_{\widetilde{X}}(h-H), \Phi(\bB_0)) \cong 
\mathbb{C}^{ 3}[-2] \\
\label{com1}
&\dR \Hom(\oO_{\widetilde{X}}(h-H), \Phi(\bB_1)) \cong 
\mathbb{C} \oplus \mathbb{C}[-2] \\
\label{com1.5}
& \dR \Hom(\Phi(\bB_0), \oO_{\widetilde{X}}(-h))^{\vee} \cong
\mathbb{C}^{ 6} \\
\label{com2}
& \dR \Hom(\Phi(\bB_1), \oO_{\widetilde{X}}(-h))^{\vee} \cong
\mathbb{C}^{ 3}.
\end{align}
\end{lem}
\begin{proof}
Since $\omega_{\wX}=\oO_{\wX}(-2H-h)$, the LHS of (\ref{com1.5}), 
(\ref{com2}) are written as
\begin{align*}
\dR \Hom(\Phi(\bB_i), \oO_{\widetilde{X}}(-h))^{\vee} \cong
\dR \Gamma(\wX, \Phi(\bB_i)(-2H)[4]). 
\end{align*}
By the exact sequence (\ref{canonical}), 
we see that $j_{\ast}\Phi(\bB_0)$ is quasi-isomorphic
to the complex
\begin{align}\label{aB0}
\left(
\begin{array}{c}
\oO_{\wP^5}(-2H') \\
\oplus \\
\oO_{\wP^5}(-2H'-h')^{\oplus 3} \\
\oplus \\
\oO_{\wP^5}(-2H'-2h')^{\oplus 3} \\
\oplus \\
\oO_{\wP^5}(-2H'-3h')
\end{array}\right)
\stackrel{\delta_0}{\to} 
\left(
\begin{array}{c}
\oO_{\wP^5}(-H')^{\oplus 3} \\
\oplus \\
\oO_{\wP^5}(-H'-h')^{\oplus 2} \\
\oplus \\
\oO_{\wP^5}(-H'-2h')^{\oplus 3} 
\end{array}\right). 
\end{align}
Applying $\otimes_{q^{\ast}\bB_0}q^{\ast} \bB_1$
to the sequence (\ref{canonical}), 
and using (\ref{Bkl}), we obtain the exact sequence
\begin{align}\label{apply} 
0 \to q^{\ast}\bB_1(-2H') \stackrel{\delta_1}{\to}
 q^{\ast}\bB_0(h'-H') \to j_{\ast} \Phi(\bB_1) \to 0
\end{align}
where $\delta_1$ is the morphism (\ref{delta}).
Hence $j_{\ast}\Phi(\bB_1)$ is quasi-isomorphic to the complex
\begin{align}\label{aB1}
\left(
\begin{array}{c}
\oO_{\wP^5}(-2H')^{\oplus 3} \\
\oplus \\
\oO_{\wP^5}(-2H'-h')^{\oplus 2} \\
\oplus \\
\oO_{\wP^5}(-2H'-2h')^{\oplus 3}
\end{array}\right)
\stackrel{\delta_1}\to 
\left(
\begin{array}{c}
\oO_{\wP^5}(-H'+h') \\
\oplus \\
\oO_{\wP^5}(-H')^{\oplus 3} \\
\oplus \\
\oO_{\wP^5}(-H'-h')^{\oplus 3} \\
\oplus \\
\oO_{\wP^5}(-H'-2h')
\end{array}\right). 
\end{align}
Applying $\otimes \oO_{\wP^5}(H'-h')$, ~$\otimes \oO_{\wP^5}(-2H')[4]$
to the complexes (\ref{aB0}), (\ref{aB1})
and then applying 
$\dR \Gamma(\wP^5, -)$, we 
see that (\ref{com0}), (\ref{com1}), (\ref{com1.5}), (\ref{com2})
are quasi-isomorphic to the following complexes respectively: 
\begin{align*}
&(M_{-1, -1} \oplus M_{-1, -2}^{\oplus 3} \oplus M_{-1, -3}^{\oplus 3}
\oplus M_{-1, -4} \to M_{0, -1}^{\oplus 3} \oplus M_{0, -2}^{\oplus 2}
 \oplus M_{0, -3}^{\oplus 3}) \\
&(M_{-1, -1}^{\oplus 3} \oplus M_{-1, -2}^{\oplus 2} \oplus 
M_{-1, -3}^{\oplus 3} \to 
M_{0, 0} \oplus M_{0, -1}^{\oplus 3} \oplus M_{0, -2}^{\oplus 3}
\oplus M_{0, -3}) \\
& (M_{-4, 0} \oplus M_{-4, -1}^{\oplus 3} \oplus M_{-4, -2}^{\oplus 3}
\oplus M_{-4, -3} \to M_{-3, 0}^{\oplus 3} \oplus M_{-3, -1}^{\oplus 2}
 \oplus M_{-3, -2}^{\oplus 3})[4] \\
&(M_{-4, 0}^{\oplus 3} \oplus M_{-4, -1}^{\oplus 2} \oplus 
M_{-4, -2}^{\oplus 3} \to 
M_{-3, 1} \oplus M_{-3, 0}^{\oplus 3} \oplus M_{-3, -1}^{\oplus 3}
\oplus M_{-3, -2})[4].  
\end{align*}
Applying the computation in Lemma~\ref{Mkl}, we obtain the result. 
\end{proof}

\subsection{Computation of $\Theta(\bB_i)$}
The purpose of this subsection is to 
compute $\Theta(\bB_i)$
for $i=0, 1$, using 
an explicit description of $\delta_i$
in Subsection~\ref{subsec:explicit}
and computations in Subsection~\ref{subsec:compute}.
Let $\Phi$ be the fully faithful embedding 
given in (\ref{PhiFM}). 
The following lemma 
includes the key computation in this subsection: 
\begin{lem}\label{lem:PhiB}
There is an isomorphism
\begin{align*}
\dR \sigma_{\ast}\Phi(\bB_1) \cong I_P \oplus 
\oO_X(-1)^{\oplus 3}.
\end{align*}
\end{lem}
\begin{proof}
By (\ref{apply}) and (\ref{easy}), the object $j_{\ast} \Phi(\bB_1)$
is quasi-isomorphic to the complex 
\begin{align}\notag
\left(
\begin{array}{c}
\oO_{\wP^5}(-2H')^{\oplus 3} \\
\oplus \\
\oO_{\wP^5}(D'-3H')^{\oplus 2} \\
\oplus \\
\oO_{\wP^5}(2D'-4H')^{\oplus 3}
\end{array}\right)
\stackrel{\delta_1}\to 
\left(
\begin{array}{c}
\oO_{\wP^5}(-D') \\
\oplus \\
\oO_{\wP^5}(-H')^{\oplus 3} \\
\oplus \\
\oO_{\wP^5}(D'-2H')^{\oplus 3} \\
\oplus \\
\oO_{\wP^5}(2D'-3H')
\end{array}\right). 
\end{align}
We apply $\dR p_{\ast}$ to the above 
complex. 
Using (\ref{Rp}) and 
\begin{align*}
\dR \sigma_{\ast}\oO_{\wP^5}(-D') \cong I_{P}'
\end{align*}
where $I_{P}' \subset \oO_{\mathbb{P}^5}$ is the 
ideal sheaf of $P$, 
we see that $i_{\ast} \dR \sigma_{\ast}\Phi(\bB_1)$ is quasi-isomorphic 
to the complex
\begin{align}\label{theta3}
\left(
\begin{array}{c}
\oO_{\mathbb{P}^5}(-2){\bf e_1} \\
\oplus \\
\oO_{\mathbb{P}^5}(-2){\bf e_2} \\
\oplus \\
\oO_{\mathbb{P}^5}(-2){\bf e_3} \\
\oplus \\
\oO_{\mathbb{P}^5}(-3){\bf f} \\
\oplus \\
\oO_{\mathbb{P}^5}(-3){\bf e_1 \wedge e_2 \wedge e_3} \\
\oplus \\
\oO_{\mathbb{P}^5}(-4){\bf e_1 \wedge e_2 \wedge f} \\
\oplus \\
\oO_{\mathbb{P}^5}(-4){\bf e_2 \wedge e_3 \wedge f} \\
\oplus \\
\oO_{\mathbb{P}^5}(-4){\bf e_3 \wedge e_1 \wedge f}
\end{array}\right)
\stackrel{p_{\ast}\delta_1}{\to}
\left(
\begin{array}{c}
I_P' \\
\oplus \\
\oO_{\mathbb{P}^5}(-1){\bf e_1 \wedge e_2} \\
\oplus \\
\oO_{\mathbb{P}^5}(-1){\bf e_2 \wedge e_3} \\
\oplus \\
\oO_{\mathbb{P}^5}(-1){\bf e_3 \wedge e_1} \\
\oplus \\
\oO_{\mathbb{P}^5}(-2){\bf e_1 \wedge f} \\
\oplus \\
\oO_{\mathbb{P}^5}(-2){\bf e_2 \wedge f} \\
\oplus \\
\oO_{\mathbb{P}^5}(-2){\bf e_3 \wedge f} \\
\oplus \\
\oO_{\mathbb{P}^5}(-3){\bf e_1 \wedge e_2 \wedge e_3 \wedge f}
\end{array}\right)
\end{align}
Here we have specified basis elements of both sides 
of (\ref{theta3})
induced from those of $\wedge^{\ast}E$. 
Since $\delta_1$ is injective, 
the morphism $p_{\ast}\delta_1$ is generically injective,
hence it is injective. 
This implies that $\dR \sigma_{\ast}\Phi(\bB_1)$ is a
coherent sheaf on $X$. 

Now we give an explicit description of $p_{\ast}\delta_1$
using the notation in Subsection~\ref{subsec:explicit}.  
Since $p_{\ast}x_7=1$, the right wedge product by (\ref{wedge*})
pushes down via $p_{\ast}$ to the right wedge product by the element
\begin{align}\label{wedgep}
x_4 {\bf e_1} + x_5 {\bf e_2} + x_6 {\bf e_3} + {\bf f}. 
\end{align}
Also by setting $\partial_i' \cneq \partial_i /2$
and 
\begin{align*}
W'' \cneq W'(x_1, x_2, x_3) + \frac{1}{2} \sum_{1\le i, j\le 3} x_i x_j W_{ij}(x_4, x_5, x_6)
\end{align*}
we have the following relations: 
\begin{align}\notag
&x_7 \partial_i' \widetilde{W}=\partial_i' W, \ 4\le i\le 6, \
x_7^2 \partial_7' \widetilde{W}=W''.  
\end{align}
Here we have used the relation (\ref{xrelation}). 
Then the right contraction by the element (\ref{contract*}) 
pushes down via $p_{\ast}$ to the right contraction by 
the element
\begin{align}\label{cont:p}
\partial_4' W {\bf e_1^{\ast}} + \partial_5' W{\bf e_2^{\ast}} + 
\partial_6' W{\bf e_3^{\ast}}
+ W'' {\bf f^{\ast}}. 
\end{align}
The morphism $p_{\ast}\delta_1$ is the 
sum of the right wedge product by (\ref{wedgep})
and the right contraction by (\ref{cont:p}). 
Therefore
if we regard local sections of both sides of (\ref{theta3}) 
as column vectors,
we see that $p_{\ast}\delta_1$ in (\ref{theta3}) is 
given by the matrix
\begin{align*}
M=
\left(
\begin{array}{cccccccc}
\Pd_4 W & \Pd_5 W & \Pd_6 W & W'' & 0 & 0 & 0 & 0 \\
x_5 & -x_4 & 0 & 0 & \Pd_6 W & W'' & 0 & 0 \\
0 & x_6 & -x_5 & 0 & \Pd_4 W & 0 & W'' & 0 \\
-x_6 & 0 & x_4 & 0 & \Pd_5 W & 0 & 0 & W'' \\
1 & 0 & 0 & -x_4 & 0 & -\Pd_5 W & 0 & \Pd_6 W \\
0 & 1 & 0 & -x_5 & 0 & \Pd_4 W & -\Pd_6 W & 0 \\
0 & 0 & 1 & -x_6 & 0 & 0 & \Pd_5 W & -\Pd_4 W \\
0 & 0 & 0 & 0 & 1 & -x_6 & -x_4 & -x_5
\end{array}
\right). 
\end{align*}
Now we define the matrices $N_1$, $N_2$ to be
\begin{align*}
N_1 &\cneq \left(
\begin{array}{cccccccc}
0 & 1 & 0 & 0 & -x_5 & x_4 & 0 & -\Pd_6 W \\
0 & 0 & 1 & 0 & 0 & -x_6 & x_5 & -\Pd_4 W \\
0 & 0 & 0 & 1 & x_6 & 0 & -x_4 & -\Pd_5 W
\end{array}  
\right) \\
N_2 &\cneq 
\left(
\begin{array}{cccccccc}
0 & 0 & 0 & 0 & 0 & 1 & 0 & 0 \\
0 & 0 & 0 & 0 & 0 & 0 & 1 & 0 \\
0 & 0 & 0 & 0 & 0 & 0 & 0 & 1
\end{array}  
\right)
\end{align*}
and define $N_3$, $N_4$ to be
\begin{align*}
N_3 \cneq
\left( \begin{array}{c}
1 \\
0 \\
0 \\
0 \\
0 \\
0 \\
0 \\
0
\end{array}
\right), \quad 
N_4 \cneq
\left( \begin{array}{c}
x_4 \\
x_5 \\
x_6 \\
1 \\
0 \\
0 \\
0 \\
0
\end{array}
\right).
\end{align*}
Then noting that
\begin{align*}
W'' + x_4 \Pd_4 W + x_5 \Pd_5 W + x_6 \Pd_6 W =W
\end{align*}
the above matrices satisfy the following relations
\begin{align*}
N_1 M= W \cdot N_2, \ 
M N_4= W \cdot N_3, \
N_2 N_4=N_1 N_3=0. 
\end{align*}
This implies that we have the commutative diagram
of sheaves on $\mathbb{P}^5$
\begin{align*}
\xymatrix{
\oO(-3) \ar[r]^{\cdot W} \ar[d]_{N_4} & I_P' \ar[d]_{N_3} \\
\oO(-2)^{\oplus 3} \oplus \oO(-3)^{\oplus 2} \oplus 
\oO(-4)^{\oplus 3}
 \ar[r]^{M} \ar[d]_{N_2} & I_P' \oplus \oO(-1)^{\oplus 3} \oplus 
\oO(-2)^{\oplus 3} \oplus \oO(-3) \ar[d]_{N_1} \\
\oO(-4)^{\oplus 3} \ar[r]^{\cdot W} & \oO(-1)^{\oplus 3}
}
\end{align*}
such that the induced sequence of sheaves on $X$
\begin{align}\label{induced}
0 \to I_P \to \dR \sigma_{\ast}\Phi(\bB_1) \to \oO_{X}(-1)^{\oplus 3} \to 0
\end{align}
is a complex. 
The above sequence is right exact since $N_1$ is surjective, 
and left exact since $N_3$ is injective and 
the cokernel of $N_4$ is locally free. 
Furthermore 
the middle cohomology of (\ref{induced}) is quasi-isomorphic to the complex 
\begin{align*}
\oO_{\mathbb{P}^5}(-2)^{\oplus 3} \oplus \oO_{\mathbb{P}^5}(-3)
\to \oO_{\mathbb{P}^5}(-2)^{\oplus 3} \oplus \oO_{\mathbb{P}^5}(-3)
\end{align*}
 which must be quasi-isomorphic to zero since 
it is a sheaf. 
Therefore (\ref{induced}) is a short exact sequence in 
$\Coh(X)$. 
Since 
$H^1(X, I_P(1))=0$ as $I_P \in \dD_X$, 
the exact sequence (\ref{induced}) 
splits and we obtain a desired isomorphism. 
\end{proof}

\begin{prop}\label{prop:ThetaB}
There is an isomorphism 
\begin{align*}
\Theta(\bB_1) \cong I_P[-1]. 
\end{align*}
\end{prop}
\begin{proof}
By (\ref{explicit:Kuz}), Lemma~\ref{lem:com2} and Lemma~\ref{lem:PhiB}, 
the object $\Theta(\bB_1)$ is written as
\begin{align}\label{ThetaB}
\{ I_P \oplus I_P[-2] \to I_P  \oplus 
\oO_X(-1)^{\oplus 3} \to \oO_X(-1)^{\oplus 3} \}. 
\end{align}
Since $I_P \in \dD_X$, $\Theta(\bB_1) \in \dD_X$ 
and $\oO_X(-1) \notin \dD_X$, 
the $\oO_X(-1)^{\oplus 3}$-component of the right 
morphism in (\ref{ThetaB}) must be an isomorphism. 
Hence we have
\begin{align*}
\Theta(\bB_1) \cong \mathrm{Cone} \left( I_P \oplus I_P[-2] 
\stackrel{(\theta, \theta')}{\lr} I_P  \right).  
\end{align*}
The morphism $\theta \colon I_P \to I_P$ must be 
non-zero, hence an isomorphism,
since otherwise $\Theta(\bB_1)$ is decomposable 
which 
contradicts to that $\bB_1$ is indecomposable 
(cf.~Lemma~\ref{lem:spherical})
and $\Theta$ is an equivalence. 
Therefore $\theta$ is an isomorphism 
and we have $\Theta(\bB_1) \cong I_P[-1]$.
\end{proof}
We will also need some computations of $\Theta(\bB_0)$. 
We first show the following:
\begin{lem}\label{lem:PhiB2}
There is an isomorphism
\begin{align*}
\dR \sigma_{\ast} \Phi(\bB_0) \cong 
I_P^{\vee}(-2) \oplus \oO_X(-1)^{\oplus 3}. 
\end{align*}
Here $-^{\vee}$ is the derived dual. 
\end{lem}
\begin{proof}
By the Grothendieck duality, we have
\begin{align}\notag
i_{\ast} \dR \hH om_X(\dR \sigma_{\ast} \Phi(\bB_0), 
i^{!}\oO_{\mathbb{P}^5}) & \cong
\dR \hH om_{\mathbb{P}^5}(i_{\ast} \dR \sigma_{\ast} \Phi(\bB_0), 
\oO_{\mathbb{P}^5}) \\
\notag
&\cong \dR \hH om_{\mathbb{P}^5}(\dR p_{\ast} j_{\ast} \Phi(\bB_0), 
\oO_{\mathbb{P}^5}) \\
\label{Grothen2}
&\cong \dR p_{\ast} \dR \hH om_{\wP^5}
(j_{\ast}\Phi(\bB_0), p^{!}\oO_{\mathbb{P}^5}). 
\end{align}
We apply $\dR \hH om_{\wP^5}(-, p^{!}\oO_{\mathbb{P}^5})$
to the complex (\ref{aB0}), and push it down 
via $\dR p_{\ast}$.
Noting $p^{!}\oO_{\mathbb{P}^5}=\oO_{\wP^5}(2D')$, 
the resulting complex becomes
\begin{align}\notag
\left(
\begin{array}{c}
\oO_{\mathbb{P}^5}(1){\bf e_1^{\ast}} \\
\oplus \\
\oO_{\mathbb{P}^5}(1){\bf e_2^{\ast}} \\
\oplus \\
\oO_{\mathbb{P}^5}(1){\bf e_3^{\ast}} \\
\oplus \\
\oO_{\mathbb{P}^5}(2){\bf f^{\ast}} \\
\oplus \\
\oO_{\mathbb{P}^5}(2){\bf e_1^{\ast} \wedge e_2^{\ast} \wedge e_3^{\ast}} \\
\oplus \\
\oO_{\mathbb{P}^5}(3){\bf e_1^{\ast} \wedge e_2^{\ast} \wedge f^{\ast}} \\
\oplus \\
\oO_{\mathbb{P}^5}(3){\bf e_2^{\ast} \wedge e_3^{\ast} \wedge f^{\ast}} \\
\oplus \\
\oO_{\mathbb{P}^5}(3){\bf e_3^{\ast} \wedge e_1^{\ast} \wedge f^{\ast}}
\end{array}\right)
\stackrel{p_{\ast}\delta_0^{\vee}}{\to}
\left(
\begin{array}{c}
\oO_{\mathbb{P}^5}(2) \\
\oplus \\
\oO_{\mathbb{P}^5}(3){\bf e_1^{\ast} \wedge e_2^{\ast}} \\
\oplus \\
\oO_{\mathbb{P}^5}(3){\bf e_2^{\ast} \wedge e_3^{\ast}} \\
\oplus \\
\oO_{\mathbb{P}^5}(3){\bf e_3^{\ast} \wedge e_1^{\ast}} \\
\oplus \\
\oO_{\mathbb{P}^5}(4){\bf e_1^{\ast} \wedge f^{\ast}} \\
\oplus \\
\oO_{\mathbb{P}^5}(4){\bf e_2^{\ast} \wedge f^{\ast}} \\
\oplus \\
\oO_{\mathbb{P}^5}(4){\bf e_3^{\ast} \wedge f^{\ast}} \\
\oplus \\
I_P'(5){\bf e_1^{\ast} \wedge e_2^{\ast} \wedge e_3^{\ast}
\wedge f^{\ast}}
\end{array}\right)
\end{align}
The morphism $p_{\ast} \delta_0^{\vee}$ is the 
sum of the right contraction by (\ref{wedgep}) and the 
right wedge product by (\ref{cont:p}). Therefore 
we can apply the exactly same computation in Lemma~\ref{lem:PhiB}, 
and show that 
\begin{align*}
\dR p_{\ast} \dR \hH om_{\wP^5}
(j_{\ast}\Phi(\bB_0), p^{!}\oO_{\mathbb{P}^5})[1]
\cong I_P(5) \oplus \oO_{X}(4)^{\oplus 3}. 
\end{align*}
Then noting $i^{!}\oO_{\mathbb{P}^5} =\oO_X(3)[-1]$, 
the above isomorphism together with (\ref{Grothen2})
yield a desired result. 
\end{proof}
As for $\Theta(\bB_0)$, we only have to 
compute its numerical class as follows: 
\begin{lem}\label{lem:ThetaB0}
We have the identity in $N(X)$:
\begin{align*}
[\Theta(\bB_0)]=[I_P^{\vee}(-2)]-3[I_P]-3[\oO_X(-1)]. 
\end{align*}
\end{lem}
\begin{proof}
The claim follows from Lemma~\ref{lem:PhiB2}, (\ref{com0}), (\ref{com1.5})
and (\ref{explicit:Kuz}). 
\end{proof}
\subsection{Evaluations at skyscraper sheaves}
Let $S$ be the K3 surface (\ref{K3:double}). 
For a point $x\in S$, the skyscraper 
sheaf $\oO_x$ determines an object in 
$\Coh(S, \alpha)$. 
In the notation of Subsection~\ref{subsec:sheaves}, we set
\begin{align*}
L_x \cneq \Upsilon(\oO_x)
\in \Coh(\bB_0). 
\end{align*}
The object $\Theta(L_x) \in \dD_X$ is studied by 
Lahoz-Macri-Stellari~\cite[Section~4]{LMS}.
They show that there is an isomorphism
\begin{align}\label{ACM}
\Theta(L_x) \cong M_x[1]
\end{align}
where $M_x$ is a rank four Gieseker 
stable ACM bundle on $X$ with 
Chern character given by
\begin{align*}
\ch(M_x)=(4, -2H, -P, l, 1/4). 
\end{align*}
Here $l$ is a class of a line in $X$. 
Furthermore there are exact sequences
(cf.~\cite[Proposition~4.2, Step~4]{LMS})
\begin{align}\label{Me1}
&0 \to \oO_X(-1)^{\oplus 2} \to M_x \to K_x \to 0, \\
\label{Me2}
&0 \to K_x \to I_P^{\oplus 2} \stackrel{\mathrm{ev}}{\to}
 I_{l_x, Q_{f(x)}} \to 0. 
\end{align}
Here $Q_{f(x)}$ is the quadric defined to be
$\sigma(\pi^{-1}f(x))$, and 
$l_{x} \subset Q_{f(x)}$ is 
 a line determined by $x$, 
and $I_{l_x, Q_{f(x)}}$ is 
the ideal sheaf of $l_x$ in $Q_{f(x)}$. 
The purpose of this subsection is to compare the 
following objects
\begin{align}\label{compare}
F_X (M_x), \quad \mathrm{ST}_{I_P}^{-1}(M_x)[1]. 
\end{align}
Here $F_X$ is defined by (\ref{def:FX}), and $\mathrm{ST}_{I_P}^{-1}$
is the inverse of the Seidel-Thomas twist 
associated to $I_P$, which is spherical by 
Proposition~\ref{prop:ThetaB}. 
We first investigate the LHS of (\ref{compare}). 
\begin{lem}
There is an isomorphism
\begin{align}\label{ev1}
F_X(M_x) \cong \mathrm{Cone} (\oO_X^{\oplus 4} \stackrel{\mathrm{ev}}{\to}
K_x(1)). 
\end{align}
\end{lem}
\begin{proof}
Applying $\otimes \oO_X(1)$
to the exact sequence (\ref{Me1}), we obtain the 
exact sequence 
\begin{align}\label{Me3}
0 \to \oO_X^{\oplus 2} \to M_x(1) \to K_x(1) \to 0. 
\end{align}
Applying $\dR \Gamma(X, -)$, we obtain the 
distinguished triangle
\begin{align}\label{GMK}
\mathbb{C}^{ 2} \to \dR \Gamma(X, M_x(1)) \to 
\dR \Gamma(X, K_x(1)). 
\end{align}
It is easy to see that 
\begin{align*}
\dR \Gamma(X, I_P(1))=\mathbb{C}^{ 3}, \ 
\dR \Gamma(X, I_{l_x, Q_{f(x)}}(1))=\mathbb{C}^{ 2}.
\end{align*}
From (\ref{Me2}), we obtain the 
distinguished triangle
\begin{align*}
\dR \Gamma(X, K_x(1)) \to \mathbb{C}^{ 6} \to \mathbb{C}^{ 2}. 
\end{align*}
By~\cite[Proposition~4.4, Step~1]{LMS}, we have 
$H^1(X, K_x(1))=0$, hence we obtain 
$\dR \Gamma(X, K_x(1))=\mathbb{C}^{ 4}$. 
Combined with (\ref{GMK}), we obtain 
$\dR \Gamma(X, M_x(1))=\mathbb{C}^{ 6}$ and 
\begin{align*}
F_X(M_x) &\cong \mathrm{Cone}(\oO_X^{\oplus 6} \to M_x(1)). 
\end{align*}
Then (\ref{ev1}) follows from the above isomorphism
and taking account of the exact sequence (\ref{Me3}). 
\end{proof}

\begin{lem}\label{lem:isomF}
The object $F_X(M_x)$ is isomorphic to the following object:
\begin{align*} 
\mathrm{Cone} (I_{l_{\iota(x)}, Q_{f(x)}}\stackrel{\rm{ev}}{\to}
\Ext^2(I_{l_{\iota(x)}, Q_{f(x)}}, \oO_X(-1))^{\vee}
\otimes \oO_X(-1)[2])[-1].  
\end{align*}
\end{lem}
\begin{proof}
By~\cite[Lemma~4.3]{LMS}, the sheaf $K_x(1)$ fits 
into the exact sequence
\begin{align*}
0 \to 
I_{P \cup Q_{f(x)}}^{\oplus 2}(1)
 \to K_x(1) \to I_{l_{\iota(x)}, Q_{f(x)}} \to 0. 
\end{align*}
Here $I_{P \cup Q_{f(x)}}$ is the ideal sheaf of 
$P \cup Q_{f(x)}$ in $X$, which is a complete intersection of 
two 
hyperplanes in $X$. 
Since $\Hom(\oO_X, I_{l_{\iota(x)}, Q_{f(x)}})=0$, the 
evaluation morphism in the RHS of (\ref{ev1})
factors through $I_{P \cup Q_{f(x)}}^{\oplus 2}(1)$. 
Moreover, the Koszul resolution of 
$\oO_{P \cup Q_{f(x)}}$ yields the exact sequence
\begin{align*}
0 \to \oO_{X}(-1)^{\oplus 2} \to \oO_X^{\oplus 4} \to 
I_{P \cup Q_{f(x)}}^{\oplus 2}(1) \to 0. 
\end{align*}
Hence we obtain the distinguished triangle
\begin{align}\label{univ:F}
\oO_X(-1)^{\oplus 2} [1] \to F_X(M_x) \to I_{l_{\iota(x)}, Q_{f(x)}}
\stackrel{\theta}{\to} \oO_X(-1)^{\oplus 2} [2].
\end{align}
On the other hand, using Serre duality, we have 
\begin{align*}
\Ext^2(I_{l_{\iota(x)}, Q_{f(x)}}, \oO_X(-1))^{\vee}
&\cong H^2(Q_{f(x)}, I_{l_{\iota(x)}, Q_{f(x)}}(-2)) \\
&\cong \mathbb{C}^2. 
\end{align*}
Therefore it is enough to show 
the morphism $\theta$ in (\ref{univ:F}) is identified with the 
evaluation morphism. 
This easily follows from the vanishing
\begin{align}\notag
\Hom(F_X(M_x), \oO_X(-1)[1])=0
\end{align}
due to $F_X(M_x) \in \dD_X$. 
\end{proof}

We next investigate the RHS of (\ref{compare}). 
\begin{lem}
There is an isomorphism
\begin{align}\label{isom:ST}
\mathrm{ST}_{I_P}^{-1}(M_x)[1] \cong \mathrm{Cone}(M_x \stackrel{\mathrm{ev}}{\to}
 I_P^{\oplus 2}). 
\end{align}
\end{lem}
\begin{proof}
Since $\Theta(L_x)=M_x[1]$ and 
$\Theta(\bB_1)=I_P[-1]$ by Proposition~\ref{prop:ThetaB}, 
and $\Theta$ is an equivalence, we have 
\begin{align}\label{MIL}
\dR \Hom(M_x, I_P) & \cong \dR \Hom(L_x, \bB_1)[2]. 
\end{align}
Under the equivalence (\ref{Eq:K3}), 
the object $L_x$ corresponds to $\oO_x$ and 
$\bB_1$ corresponds to a rank two $\alpha$-twisted
vector bundle on $S$. Therefore we see that (\ref{MIL})
is isomorphic to $\mathbb{C}^{ 2}$. 
Hence (\ref{isom:ST}) follows from the definition of 
$\mathrm{ST}_{I_P}^{-1}$. 
\end{proof}

\begin{lem}\label{lem:isomS}
The object $\mathrm{ST}_{I_P}^{-1}(M_x)[1]$ is isomorphic 
to the following object
\begin{align*}
\mathrm{Cone} (I_{l_{x}, Q_{f(x)}}\stackrel{\rm{ev}}{\to}
\Ext^2(I_{l_{x}, Q_{f(x)}}, \oO_X(-1))^{\vee}
\otimes \oO_X(-1)[2])[-1]. 
\end{align*}
\end{lem}
\begin{proof}
From the exact sequences (\ref{Me1}) and (\ref{Me2}),
we have the morphisms
\begin{align*}
M_x \twoheadrightarrow K_x \hookrightarrow I_P^{\oplus 2}.
\end{align*}
The above composition must coincide with the 
evaluation morphism in the RHS of 
(\ref{isom:ST}) up to a base change of $I_P^{\oplus 2}$, 
since otherwise
its image has rank less than or equal to one which 
contradicts to that $K_x$ has rank two. 
Therefore 
we obtain a commutative diagram
\begin{align*}
\xymatrix{
& & \mathrm{ST}_{I_P}^{-1}(M_x) \ar[d] & I_{l_x, Q_x}[-1] \ar[d] & \\
0 \ar[r] & \oO_X(-1)^{\oplus 2} \ar[r] & M_x \ar[r]\ar[d]^{\rm{ev}}
 & K_x \ar[r]\ar[d]
 & 0 \\
& & I_P^{\oplus 2} \ar[r]^{\cong} & I_P^{\oplus 2} &
}.
\end{align*} 
As a result, we obtain the distinguished triangle
\begin{align}
\oO_X(-1)^{\oplus 2}[1] \to \mathrm{ST}_{I_P}^{-1}(M_x)[1] \to 
I_{l_x, Q_x} \stackrel{\theta'}{\to}
\oO_X(-1)^{\oplus 2}[2]. 
\end{align}
Similarly to the proof of Lemma~\ref{lem:isomF},
the morphism $\theta'$ is identified with the 
evaluation morphism because of the vanishing
\begin{align*}
\Hom(\mathrm{ST}_{I_P}^{-1}(M_x)[1], \oO_X(-1)[1])=0
\end{align*}
due to $\mathrm{ST}_{I_P}^{-1}(M_x) \in \dD_X$. 
\end{proof}
As a corollary of Lemma~\ref{lem:isomF} and 
Lemma~\ref{lem:isomS}, we obtain the 
following: 
\begin{cor}\label{cor:F=ST}
For any $x\in S$, there is an isomorphism
\begin{align*}
F_X(M_x) \cong \mathrm{ST}_{I_P}^{-1}(M_{\iota(x)})[1]. 
\end{align*}
\end{cor}

\subsection{Proof of Proposition~\ref{key:prop}}\label{subsec:proof}
\begin{proof}
Let us consider the autequivalence $\Psi$ of $D^b \Coh(S, \alpha)$
obtained as the composition
\begin{align*}
\Psi \colon 
D^b \Coh(S, \alpha) &\stackrel{\Upsilon}{\to} D^b \Coh(\bB_0) 
\stackrel{\Theta}{\to} \dD_X \stackrel{F_X}{\to} \dD_X \\
&\stackrel{\Theta^{-1}}{\to} D^b \Coh(\bB_0) \stackrel{F_B^{-1}}{\to}
D^b \Coh(\bB_0) \stackrel{\Upsilon^{-1}}{\to} D^b \Coh(S, \alpha).
\end{align*}
It is enough to show that there is an isomorphism of functors
\begin{align}\label{enough:ten}
\Psi (-) \cong \id_{D^b \Coh(S, \alpha)}.  
\end{align}
By applying the results so far, 
we show the following: 
\begin{lem}\label{num:1}
There is $\lL \in \Pic(S)$ and an 
 isomorphism of functors
\begin{align*}
\Psi \cong \otimes 
\lL.
\end{align*} 
\end{lem}
\begin{proof}
For $x\in S$, we have
\begin{align}\notag
\Psi(\oO_x) &= \Upsilon^{-1} \circ \otimes_{\bB_0} \bB_1 \circ
\mathrm{ST}_{\bB_1} \circ \Theta^{-1} \circ F_X
\circ \Theta(L_x)[-1] \\
\label{iisom1}
&\cong \Upsilon^{-1} \circ \otimes_{\bB_0} \bB_{1} \circ \Theta^{-1} \circ
\mathrm{ST}_{I_P[-1]} \circ F_X(M_x) \\
\label{iisom2}
&\cong \Upsilon^{-1}(L_{\iota(x)}\otimes_{\bB_0} \bB_1) \\
\label{iisom3}
&\cong \oO_x. 
\end{align}
Here we have used (\ref{ACM}) and 
Proposition~\ref{prop:ThetaB}
in (\ref{iisom1}), Proposition~\ref{cor:F=ST}
and an obvious
 fact $\mathrm{ST}_{I_P}=\mathrm{ST}_{I_P[-1]}$
in (\ref{iisom2}), and 
the fact that $\otimes_{\bB_0} \bB_1$ takes
$L_x$ to $L_{\iota(x)}$ in (\ref{iisom3}) 
which is a well-known property of 
representations of Clifford algebras~\cite[Corollary~2.4.5]{Mor}. 
On the other hand, by~\cite{Ca-St},
there is an object $\pP$ in $D^b \Coh(S\times S, \alpha^{-1} \boxtimes \alpha)$
such that  
the equivalence $\Psi$
is of Fourier-Mukai type
with kernel $\pP$.  
By a spectral sequence 
argument 
as in~\cite[Lemma~4.3]{Br2},
the 
condition $\Psi(\oO_x) \cong \oO_x$
for any $x\in S$ implies that 
$\pP \in \Coh(S\times S, \alpha^{-1} \boxtimes \alpha)$, 
which is flat over the first factor
and supported on the diagonal. 
In particular, $\Psi$ preserves $\Coh(S, \alpha)$, 
hence the proof of~\cite[Corollary~5.3]{Ca-St}
shows that 
$\Psi$ is written as a desired form. 
\end{proof}

By Lemma~\ref{num:1}, 
an isomorphism (\ref{enough:ten}) follows once 
we show an isomorphism $\lL \cong \oO_X$. 
Since $S$ is a K3 surface, the isomorphism classes of line bundles 
are determined by their first Chern classes. 
Hence it is enough to find a twisted vector bundle 
$\uU \in \Coh(S, \alpha)$
such that $\Psi(\uU)$ and $\uU$ 
have the same numerical classes. 
We check this for a rank two 
twisted vector bundle which corresponds to $\bB_1$
under $\Upsilon$, using the following lemma: 
\begin{lem}\label{num:3}
We have the identity in $N(\dD_X)$
\begin{align}\label{idn1}
[F_X(I_P)]=[\Theta \circ \mathrm{ST}_{\bB_1}^{-1}(\bB_0)]. 
\end{align}
\end{lem}
\begin{proof}
It is easy to see that $\dR \Hom(\oO_X, I_P(1))$
is isomorphic to $\mathbb{C}^3$, hence 
the LHS of (\ref{idn1}) is given by
\begin{align}\label{F:LHS}
[F_X(I_P)]=[I_P(1)]-3[\oO_X]
\end{align}
in $N(X)$. 
Using Lemma~\ref{lem:spherical},
Proposition~\ref{prop:ThetaB}, 
and Lemma~\ref{lem:ThetaB0}, the RHS of (\ref{idn1})
is computed in $N(X)$ as 
\begin{align}\notag
[\Theta \circ \mathrm{ST}_{\bB_1}^{-1}(\bB_0)] &=
[\Theta(\bB_0)]-3[\Theta(\bB_1)] \\
\label{comp:nume}
&=[I_P^{\vee}(-2)]-3[\oO_X(-1)]. 
\end{align}
By a standard calculation, 
the RHS of (\ref{F:LHS}) and (\ref{comp:nume}) 
have the same Chern characters given by
\begin{align*}
\left( -2, H, \frac{1}{2}H^2 -P, -\frac{1}{2}l, -\frac{1}{8}  \right)
\in H^4(X, \mathbb{Q}). 
\end{align*}
Here $l$ is a line in $X$. By the Riemann-Roch theorem on $X$, we obtain 
the identity (\ref{idn1}). 
\end{proof}

Let $\uU_1 \in \Coh(S, \alpha)$ be
the rank two twisted vector bundle which corresponds to $\bB_1$
under $\Upsilon$. Applying Proposition~\ref{prop:ThetaB} and 
Lemma~\ref{num:3}, we have 
the identities of the numerical classes of objects in $D^b \Coh(S, \alpha)$
\begin{align*}
[\Psi (\uU_1)] &= -[\Upsilon^{-1} \circ \otimes_{\bB_0}\bB_1
\circ \mathrm{ST}_{\bB_1} 
\circ  \Theta^{-1} \circ F_X \circ \Theta (\bB_1) ] \\
& = [\Upsilon^{-1} \circ \otimes_{\bB_0}\bB_1 \circ \mathrm{ST}_{\bB_1} 
\circ \Theta^{-1} \circ F_X(I_P) ] \\
&=[\uU_1].
\end{align*}
Therefore the first Chern class of 
the line bundle $\lL$ in Lemma~\ref{num:1}
is trivial.
Hence $\lL$ is trivial, and we obtain a desired 
isomorphism (\ref{enough:ten}).  
\end{proof}

\section{Construction of a Gepner type stability condition}\label{sec:const}
In this section, we prove Proposition~\ref{intro:propZG}
and Theorem~\ref{thm:introGepner}. We assume that we are in the 
same situation as in the previous section. 
\subsection{Description of $Z_G$ in terms of sheaves of Clifford algebras}
In this subsection,
we investigate 
the numerical Grothendieck
group $N(\bB_0)$ of $D^b \Coh(\bB_0)$, 
and describe the central charge $Z_G$
on $\HMF(W)$ in terms of $N(\bB_0)$. 
Let
\begin{align*}
V \subset N(\bB_0)_{\mathbb{Q}}
\end{align*}
be the $\mathbb{Q}$-vector subspace generated by all
$[\bB_k]$ for $k\in \mathbb{Z}$. 
Let us consider the autequivalence
$F_B$ in (\ref{eq:FB}), 
and its action $F_{B\ast}$ on $N(\bB_0)$. 
Obviously $F_{B\ast}$ preserves the subspace $V$. 
We compute the action of
\begin{align*}
F_{B\ast}^{-1}=-(\otimes_{\bB_0}\bB_1)_{\ast} \circ \mathrm{ST}_{\bB_1 \ast}
\colon V \to V.
\end{align*} 
The following lemma is obvious
from (\ref{Bkl}): 
\begin{lem}\label{lem:Faction}
The action $F_{B\ast}^{-1}$ on $V$ is given by 
\begin{align*}
F_{B\ast}^{-1}([\bB_i])=-[\bB_{i+1}]+\chi(B_1, B_{i+1})[B_2]. 
\end{align*}
\end{lem}
We are going to describe the action $F_{B\ast}^{-1}$
more precisely by finding a basis of $V$. 
\begin{prop}\label{prop:V}
The vector space $V$ is three dimensional, 
and
\begin{align*}
V=\mathbb{Q}[\bB_0] \oplus \mathbb{Q}[\bB_1] \oplus 
\mathbb{Q}[\bB_2]. 
\end{align*}
\end{prop}
\begin{proof}
We divide the proof into four steps. 
\begin{step}
We have $3\le \dim V\le 6$. 
\end{step}
Because $\bB_{i+2}=\bB_i(1)$ and $N(\mathbb{P}^2)$
is generated by $\oO_{\mathbb{P}^2}(i)$
for $0\le i\le 2$, 
the vector space $V$ is at least 
generated by $[\bB_i]$ for 
$0\le i\le 5$. In particular, we have 
$\dim V\le 6$. 
On the other hand, the Chern characters of $\bB_i$
for $0\le i\le 2$ as $\oO_{\mathbb{P}^2}$-modules
are given as follows: 
\begin{align}\label{chBB}
&\ch(\bB_0)=(8, -12, 12) \\
\notag
&\ch(\bB_1)=(8, -8, 7) \\
\notag
&\ch(\bB_2)=(8, -4, 4). 
\end{align} 
They are linearly independent, so 
$[\bB_i]$ for $0\le i\le 2$ are also linearly independent 
in $N(\bB_0)_{\mathbb{Q}}$. 
In particular, we have $\dim V \ge 3$. 
Below we reduce the number of generators by finding 
three more
relations among $[\bB_i]$ for $0\le i\le 5$. 
\begin{step}
First relation. 
\end{step}
For $x\in \mathbb{P}^2$, the objects $\bB_i|_{x}$
do not depend on $i$ since they correspond to $\oO_{f^{-1}(x)}^{\oplus 2}$
under the equivalence (\ref{Upsilon}). 
Therefore by taking the Koszul resolution
\begin{align}\label{Koszul:B}
0 \to \bB_i \to \bB_i(1)^{\oplus 2} \to \bB_i(2) \to \bB_i|_{x} \to 0
\end{align}
we obtain 
the following relation
\begin{align}\label{relation1}
[\bB_4]-2[\bB_2]+[\bB_0]=
[\bB_5]-2[\bB_3]+[\bB_1]. 
\end{align}
\begin{step}
Second relation. 
\end{step}
Let us take a general line $l\subset \mathbb{P}^2$, 
and consider the non-commutative scheme
$(l, \bB_0|_{l})$. 
Similarly to (\ref{Upsilon}), there is an equivalence
\begin{align*}
\Coh(\bB_0|_{l}) \cong \Coh(\bB_S|_{f^{-1}(l)}). 
\end{align*}
Since $f^{-1}(l)$ is a curve, 
the Azuyama algebra $\bB_S|_{f^{-1}(l)}$ splits 
and the RHS is equivalent to $\Coh(f^{-1}(l))$. 
This fact easily 
implies that the numerical class of an object in $D^b \Coh(\bB_0|_{l})$
is determined  by its rank and degree as $\oO_l$-module. 
Note that $\bB_i|_{l}$ all have rank eight, and 
\begin{align*}
\deg (\bB_{i+1}|_{l}) - \deg (\bB_i|_{l}) =4. 
\end{align*}
Also the object $\bB_i|_{x}$ has rank zero and degree eight for $x\in l$. 
Therefore we have the following relation in $N(\bB_0|_{l})$: 
\begin{align*}
2([\bB_3|_l]-[\bB_2|_l])=[\bB_4|_x]. 
\end{align*}
By pushing forward to $N(\bB_0)$, and taking the Koszul 
resolution (\ref{Koszul:B}) and  
the exact sequence
\begin{align*}
0 \to \bB_i(-1) \to \bB_i \to \bB_i|_{l} \to 0
\end{align*}
we obtain the relation
\begin{align}\notag
2([\bB_3]-[\bB_1])-2([\bB_2]-[\bB_0])=
[\bB_4]-2[\bB_2]+[\bB_0]. 
\end{align}
The above relation is equivalent to 
\begin{align}\label{relation2}
[\bB_4]=[\bB_0]-2[\bB_1]+2[\bB_3]. 
\end{align}
\begin{step}
Third relation.
\end{step}
We now apply Lemma~\ref{lem:Faction}
to describe $F_{B\ast}^{-1}$
in terms of $[\bB_i]$ for $0\le i\le 3$. 
Using the computation in Lemma~\ref{lem:spherical}, 
it is straightforward to deduce that
\begin{align}\label{Fdeduce}
&F_{B\ast}^{-1}([\bB_0])=
-[\bB_1] + 3[\bB_2] \\
\notag
&F_{B\ast}^{-1}([\bB_1])=
[\bB_2] \\
\notag
&F_{B\ast}^{-1}([\bB_2])=
3[\bB_2] -[\bB_3] \\
\notag
&F_{B\ast}^{-1}([\bB_3])=
-[\bB_0] +2[\bB_1]+6[\bB_2]-2[\bB_3].
\end{align} 
Here we have used the relation (\ref{relation2}) 
in the last equation. 
Applying the above formulas 
three times, 
a little computation shows that: 
\begin{align*}
F_{B\ast}^{-3}([\bB_0])=3[\bB_0]-6[\bB_1]+6[\bB_2]-2[\bB_3].
\end{align*}
By Corollary~\ref{cor:shift},
the above class should coincide with $[\bB_0]$. 
Therefore we obtain the relation
\begin{align}\label{relation3}
[\bB_3]=[\bB_0]-3[\bB_1]+3[\bB_2]. 
\end{align}
The relations (\ref{relation1}), (\ref{relation2}) and (\ref{relation3})
show that $V$ is spanned by $[\bB_i]$ for $0\le i\le 2$. 
\end{proof}
The proof of the above proposition 
also specifies the action of $F_{B\ast}^{-1}$ on $V$: 
\begin{cor}
The action of $F_{B\ast}^{-1}$ on $V$ is 
given as follows: 
\begin{align}\label{act:F}
&F_{B\ast}^{-1}([\bB_0])=
-[\bB_1] + 3[\bB_2] \\
\notag
&F_{B\ast}^{-1}([\bB_1])=
[\bB_2] \\
\notag
&F_{B\ast}^{-1}([\bB_2])=-[\bB_0]+3[\bB_1]. 
\end{align}
\end{cor}
\begin{proof}
The action (\ref{act:F}) is given by 
substituting (\ref{relation3}) into (\ref{Fdeduce}). 
\end{proof}
The following corollary will be useful in a later computation: 
\begin{cor}\label{cor:relation4}
The following relation holds in $V$: 
\begin{align}\notag
[\bB_1]=\frac{3}{8}[\bB_0] +\frac{3}{4}[\bB_2]-\frac{1}{8}[\bB_4]. 
\end{align}
\end{cor}
\begin{proof}
The claim follows from 
relations (\ref{relation2}) and (\ref{relation3}). 
\end{proof}
Now we describe the 
central charge $Z_G$ on $\HMF(W)$
in terms of $D^b \Coh(\bB_0)$. 
Let $Z_G'$ be the central charge on $D^b \Coh(\bB_0)$, 
defined to be the composition
\begin{align*}
Z_G' \colon 
 N(\bB_0) \stackrel{\Theta_{\ast}}{\to}
 N(\dD_X) \stackrel{\Phi_{1\ast}^{-1}}{\to} N(\HMF(W)) \stackrel{Z_G}{\to} 
\mathbb{C}.
\end{align*}
We compute $Z_G'$ using Corollary~\ref{cor:shift}. 
Below, we set
\begin{align*}
\omega \cneq e^{2\pi \sqrt{-1}/3} \in \mathbb{C}^{\ast}. 
\end{align*}
\begin{prop}\label{lem:uV}
There is a non-zero constant $c\in \mathbb{C}^{\ast}$ such that
the central charge $Z_G'$ is written as
\begin{align}\label{Zchi}
Z_G'(E)=c \cdot \chi(u, E)
\end{align}
where $u\in V$ is given by
\begin{align*}
u &\cneq [\bB_0] +(\omega-2)[\bB_1] -\omega[\bB_2]. 
\end{align*}
\end{prop}
\begin{proof}
By Lemma~\ref{lem:later}
and Corollary~\ref{cor:shift}, 
the central charge $Z_G'$ is 
written as $(\ref{Zchi})$
for some $u \in N(\bB_0)$
satisfying
$F_{B\ast}^{-1}u=\omega \cdot u$. 
We first show that $u \in V$
holds. 
Let us consider the decomposition
\begin{align*}
N(\bB_0)_{\mathbb{Q}}=V \oplus V^{\perp}. 
\end{align*}
Here $V^{\perp}$ is the orthogonal complement of 
$V$ with respect to $\chi(\ast, \ast)$. 
Obviously, the 
action of $F_{B\ast}^{-1}$ preserves both of $V_{\mathbb{C}}$
 and $V^{\perp}_{\mathbb{C}}$. 
Suppose by a contradiction that there is $u'\in V^{\perp}_{\mathbb{C}}$
with $F_{B\ast}^{-1}u'=\omega \cdot u'$. 
Then we have
$\chi(u', \bB_1)=0$, which implies that
\begin{align*}
F_{B\ast}^{-1}(u')=-u' \otimes_{\bB_0} \bB_1.
\end{align*}
By applying the above identity twice, 
we obtain the identity
\begin{align*}
u'\otimes_{\bB_0} \bB_0(1)=\omega^2 \cdot u'.
\end{align*}
On the other hand, the equivalence 
$\otimes_{\bB_0} \bB_0(1)$
corresponds to tensoring $\oO_{S}(h)$
on $D^b \Coh(\bB_S)$ 
in Subsection~\ref{subsec:sheaves}, 
which acts on $H^{\ast}(S, \mathbb{Z})$
by
multiplying $e^h$. 
Since this action is unipotent, there is 
no non-zero eigenvector in $H^{\ast}(S, \mathbb{C})$
with eigenvalue $\omega^2$, 
which is a contradiction.  

By the above argument and Lemma~\ref{lem:uV}, 
$u$ is written as
\begin{align*}
u=x_0[\bB_0] + x_1[\bB_1] + x_2[\bB_2]
\end{align*}
for some $x_i \in \mathbb{C}$. 
By (\ref{act:F}), the condition 
$F_{B\ast}^{-1}u=\omega \cdot u$
is given by 
\begin{align*}
\left( \begin{array}{ccc}
0 & 0 & -1 \\
-1 & 0 & 3 \\
3 & 1 & 0 
\end{array}  \right)
\left( \begin{array}{c}
x_0 \\
x_1 \\ 
x_2
\end{array}
\right)
=\omega 
\left( \begin{array}{c}
x_0 \\
x_1 \\ 
x_2
\end{array}
\right). 
\end{align*}
It has the one dimensional solution space, spanned 
by 
\begin{align*}
x_0=1, \ x_1=\omega-2, \ x_2=-\omega. 
\end{align*} 
\end{proof}

\subsection{Description of $Z_G$ in terms of twisted K3 surfaces}\label{subsec:twisted}
In this subsection, we describe the 
central charge $Z_G$ 
in terms of $\alpha$-twisted sheaves
on the K3 surface $S$. 
We first recall
the twisted Chern character theory on $D^b \Coh(S, \alpha)$, 
developed by~\cite{HuSt}. 
In our situation (cf.~\cite[Section~6]{Kuz2}), 
it depends on an additional 
 choice of the following data
\begin{align*}
B\in H^2 (S, \frac{1}{2}\mathbb{Z}), \ 
\alpha= \exp(B^{0, 2}).
\end{align*}
Here $B^{0, 2}$ means the $(0, 2)$-part 
in the Hodge decomposition of $H^2(S, \mathbb{C})$. 
By~\cite[Corollary~2.4]{HuSt}, there exists a map
\begin{align*}
\ch^{B} \colon D^b \Coh(S, \alpha) \to H^{\ast}(S, \mathbb{Z})
\end{align*}
whose image coincides with 
\begin{align*}
\widetilde{H}^{1, 1}(S, B, \mathbb{Z}) \cneq 
e^{B} \left(\bigoplus_{i=0}^{2} H^{i, i}(S, \mathbb{Q})\right) \cap
 H^{\ast}(S, \mathbb{Z}). 
\end{align*}
satisfying the Riemann-Roch theorem:
for any $E, F \in D^b \Coh(S, \alpha)$, 
we have the formula
\begin{align}\label{RR:twist}
\chi(E, F)=-\langle v^B(E), v^B(F) \rangle. 
\end{align}
Here 
$v^B(E)$ is the twisted Mukai vector
\begin{align*}
v^B(E) \cneq \ch^B(E) \sqrt{\td_S} \in
\widetilde{H}^{1, 1}(S, B, \mathbb{Z})
\end{align*}
and
$\langle -, - \rangle$ is the Mukai pairing
on $\widetilde{H}^{1, 1}(S, B, \mathbb{Z})$,
\begin{align*}
\langle (\xi_0, \xi_1, \xi_2), (\xi_0', \xi_1', \xi_2') \rangle
=\xi_1 \xi_1'-\xi_0 \xi_2' -\xi_2 \xi_0'.
\end{align*}
In particular, the map $\ch^B$ induces the isomorphism
\begin{align}\label{isom:ch}
\ch^{B} \colon N(S, \alpha) \stackrel{\cong}{\to}
\widetilde{H}^{1, 1}(S, B, \mathbb{Z}).
\end{align}
Here $N(S, \alpha)$ is the numerical 
Grothendieck group of $D^b \Coh(S, \beta)$. 

Let $\uU_i \in \Coh(S, \alpha)$ be the
twisted sheaves which correspond to $\bB_i$ under 
the equivalence $\Upsilon$. 
We prepare the following lemma:
\begin{lem}\label{lem:chU}
We can write $v^B(\uU_i)$ as 
\begin{align}\label{chU}
v^B(\uU_{i})=e^{hi/2}\left( 2, \beta, \frac{1}{4}\beta^2 +\frac{1}{2} \right)
\end{align}
for some $\beta \in H^2(S, \mathbb{Z})$ 
with $\beta-2B \in H^{1, 1}(S, \mathbb{Z})$. 
\end{lem}
\begin{proof}
Since $\bB_{i+2}=\bB_i(1)$, 
we have $\uU_{i+2}=\uU_i(h)$. 
Therefore it is enough to show 
the case of $i=0$ and $i=1$. 
In the case of $i=0$, 
we have $\chi(\uU_{0}, \uU_{0})=2$
by Lemma~\ref{lem:spherical}. 
Hence we obtain a desired form (\ref{chU})
by the Riemann-Roch theorem (\ref{RR:twist}). 
Moreover the class
$\beta-2B$ is algebraic 
since
$e^{-B} \ch^B(\uU_0)$ is algebraic.
In the case of $i=1$, using
Corollary~\ref{cor:relation4}, we
have 
\begin{align*}
v^B(\uU_1)&=\frac{3}{8}v^B(\uU_0)+\frac{3}{4}v^B(\uU_2)-\frac{1}{8}v^B(\uU_4) \\&=\left( \frac{3}{8} + \frac{3}{4}e^h - \frac{1}{8} e^{2h} \right) v^B(\uU_0) \\&=e^{h/2} v^B(\uU_0). 
\end{align*}
\end{proof}
Now we define the \textit{untwisted} Chern character
on $\alpha$-twisted sheaves to be
\begin{align}\label{untwist}
\ch(-) \cneq e^{-B} \ch^B(-) \colon 
N(S, \alpha) \to \bigoplus_{i=0}^{2} H^{i, i}(S, \mathbb{Q}). 
\end{align}
The above untwisted 
Chern character
may depend on a choice of $B$. 
If $\alpha=1$, we can take $B=0$
and then it coincides with the usual Chern character. 
The benefit of the untwisted Chern character is that it takes values in 
algebraic classes, although it may not be defined in the integer coefficient. 
We describe $Z_G$ in terms of $D^b \Coh(S, \alpha)$
as an integral which appeared in~\cite{Brs2}, 
using the untwisted 
Chern character and an algebraic $\mathfrak{B}$-field.
\begin{prop}\label{thm:int}
There is an element $\mathfrak{B} \in H^{1, 1}(S, \mathbb{Q})$
such that the composition
\begin{align}\label{compose:Z''}
N(S, \alpha) \stackrel{\Upsilon_{\ast}}{\to}
 N(\bB_0) \stackrel{\Theta_{\ast}}{\to}
 N(\dD_X) \stackrel{\Phi_{1\ast}^{-1}}{\to} N(\HMF(W)) \stackrel{Z_G}{\to} 
\mathbb{C}
\end{align}
coincides with the following integral
\begin{align}\label{integral}
Z_G''(E) \cneq 
-\int_S e^{\mathfrak{B}-\frac{\sqrt{-3}}{4}h} \ch(E) \sqrt{\td_S}. 
\end{align}
up to a non-zero scalar multiplication. 
\end{prop}
\begin{proof}
For $E \in D^b \Coh(S, \alpha)$, 
the composition (\ref{compose:Z''}) becomes 
\begin{align}\label{chiU}
Z_G'(\Upsilon_{\ast}(E)) &=
c \cdot \chi(\Upsilon_{\ast}^{-1}u, E) \\
\notag
&=c \cdot \chi\left([\uU_0] +
(\omega-2)[\uU_1]-\omega [\uU_2], E  \right). 
\end{align}
Using Lemma~\ref{lem:chU}, we have  
\begin{align*}
&v^B \left( 
[\uU_0] +
(\omega-2)[\uU_1]-\omega [\uU_2] \right) \\
&=\left\{ 1 +(\omega-2)e^{h/2} -\omega e^h \right\} v^B(\uU_0) \\
&= \left(-1, -\left(\frac{1}{2}\omega +1\right)h, -\frac{3}{4}\omega -\frac{1}{2} \right) \left(2, \beta, \frac{\beta^2}{4} +\frac{1}{2} \right) \\
&=\left(-2, -\beta-(2+\omega)h, -\frac{\beta^2}{4}-\left(\frac{1}{2}\omega+1 \right)\beta h -\frac{3}{2}\omega -\frac{3}{2}  \right) \\
&=
-2 \exp\left( \frac{1}{2}\beta + \left( \frac{1}{2}\omega +1 \right)h
 \right)
\end{align*}
Applying the Riemann-Roch theorem (\ref{RR:twist}), 
we see that the RHS of (\ref{chiU}) is written as an integral
\begin{align*}
-2c \cdot \int_S e^{-3h/4 -\beta/2 -\sqrt{-3}h/4} v^B(E). 
\end{align*} 
We define $\mathfrak{B}$ to be
\begin{align*}
\mathfrak{B} \cneq B -\frac{3}{4}h-\frac{1}{2}\beta. 
\end{align*}
Note that $\mathfrak{B}$ is an element in 
$H^{1, 1}(S, \mathbb{Q})$ 
by Lemma~\ref{lem:chU}. 
Combined with the definition of the untwisted 
Chern character (\ref{untwist}), we arrive at the desired result.
\end{proof}

\subsection{Construction of a Gepner type stability condition}
We finish the proof of Theorem~\ref{thm:introGepner}, 
hence Theorem~\ref{intro:mainthm},
in this subsection.
For $E \in D^b \Coh(S, \alpha)$, let 
$v^{\mathfrak{B}}(E)$ be the $\mathfrak{B}$-twisted Mukai vector 
\begin{align*}
v^{\mathfrak{B}}(E) & \cneq e^{\mathfrak{B}}\ch(E) \sqrt{\td_S}. 
\end{align*}
It is useful to rewrite the integral (\ref{integral})
into the following form: 
\begin{align}\label{rewrite}
Z_G''(E)=-v_2^{\mathfrak{B}}(E) + \frac{3}{16}v_0^{\mathfrak{B}}(E)
+ \frac{\sqrt{-3}}{4}v_1^{\mathfrak{B}}(E)h. 
\end{align}
Let us consider the following slope function on $\Coh(S, \alpha)$
\begin{align*}
\mu(E) \cneq \frac{v_1^{\mathfrak{B}}(E)\cdot h}{\rank(E)}.
\end{align*} 
Here we set $\mu(E)=\infty$ if $\rank(E)=0$. 
\begin{defi}
An object $E \in \Coh(S, \alpha)$ is 
$\mu$-(semi)stable if for any exact sequence
$0 \to F \to E \to G \to 0$
in $\Coh(S, \alpha)$, we have
$\mu(F) <(\le) \mu(G)$. 
\end{defi}
\begin{rmk}
When $\alpha=1$, the 
above $\mu$-stability coincides with 
the
classical twisted slope stability. 
It also behaves well even 
when $\alpha \neq 1$,
e.g. the Harder-Narasimhan property. 
The proof is the same as in the $\alpha=1$ case. 
\end{rmk}
Following~\cite{Brs2}, we define the 
following subcategories in $\Coh(S, \alpha)$
\begin{align*}
\tT &\cneq \langle E \in \Coh(S, \alpha) : 
E \mbox{ is } \mu \mbox{-semistable with } \mu(E)>0 \rangle_{\rm{ex}} \\
\fF &\cneq \langle E \in \Coh(S, \alpha) : 
E \mbox{ is } \mu \mbox{-semistable with } \mu(E)\le 0 \rangle_{\rm{ex}}.
\end{align*}
Here $\langle - \rangle_{\rm{ex}}$ is the extension closure. 
The existence of Harder-Narasimhan filtrations in $\mu$-stability
shows that the above pair is a torsion pair~\cite{HRS}
on $\Coh(S, \alpha)$. 
The associated tilting is defined to be
\begin{align*}
\aA_G \cneq \langle \fF[1], \tT \rangle_{\rm{ex}} \subset
D^b \Coh(S, \alpha). 
\end{align*}
Let $F_S$ be the autequivalence of $D^b \Coh(S, \alpha)$
defined to be
\begin{align*}
F_S \cneq \Upsilon^{-1} \circ F_B \circ \Upsilon \colon 
D^b \Coh(S, \alpha)\stackrel{\sim}{\to}
D^b \Coh(S, \alpha). 
\end{align*}
We would like to claim that $(Z_G'', \aA_G)$ is a 
Gepner type stability 
condition with respect to $(F_S, 2/3)$. 
Unfortunately we are able to prove this 
only for $\alpha \neq 1$ case. 
We note that for a general cubic fourfold 
containing a plane $P$ (e.g. when 
the numerical classes of
codimension two algebraic cycles are 
spanned by $H^2$ and $P$) 
the associated Brauer group 
satisfies $\alpha \neq 1$
(cf.~\cite[Proposition~4.8]{Kuz2}). 
The $\alpha \neq 1$ condition is required to show 
the following lemmas: 
\begin{lem}\label{lem:tMu}
Suppose that $\alpha \neq 1$. 
Then for any $E \in D^b \Coh(S, \alpha)$, 
we have 
\begin{align*}
v_0^{\mathfrak{B}}(E) \in 2\mathbb{Z}, \
2v_1^{\mathfrak{B}}(E) \in H^{1, 1}(S, \mathbb{Z}), \ 
8v_2^{\mathfrak{B}}(E) \in \mathbb{Z}. 
\end{align*}
Furthermore $4v_2^{\mathfrak{B}}(E) \notin \mathbb{Z}$ if 
$v_0^{\mathfrak{B}}(E)/2$ is odd. 
\end{lem}
\begin{proof}
Suppose that $v_0^{\mathfrak{B}}(E)=\rank(E)$ is odd. 
Then $\alpha$ is 
also the Brauer class of the twisted line bundle
$\det(F)$, whose transition function 
provides a cocycle which makes 
$\alpha$ to be trivial.
This is a contradiction, hence $v_0^{\mathfrak{B}}(E)$
is an even number. 
(Also see~\cite[Corollary~3.2]{StMa}.)
By the definition of $v^{\mathfrak{B}}(E)$, we have
\begin{align*}
v^{\mathfrak{B}}(E)&= e^{-3h/4 -\beta/2} \cdot v^B(E).
\end{align*}
Hence if we 
write $\xi_i=v^B(E) \in H^{2i}(S, \mathbb{Z})$, 
then 
\begin{align}\label{v11}
&v_1^{\mathfrak{B}}(E)=\xi_1-\left(\frac{3}{4}h +\frac{\beta}{2} \right)\xi_0 \\\notag
 &v_2^{\mathfrak{B}}(E)=
\left(\frac{\beta^2}{4}+\frac{3}{4}\beta h + \frac{9}{8}
  \right)\frac{\xi_0}{2} -\left(\frac{3}{4}h+ \frac{\beta}{2}\right)\xi_1
+\xi_2. 
\end{align}
Therefore the claim for $v_1^{\mathfrak{B}}(E)$
and $v_2^{\mathfrak{B}}(E)$
 follow from the integrality of $\beta$, $\xi_i$
and $v_0^{\mathfrak{B}}(E)=\xi_0 \in 2\mathbb{Z}$. 
\end{proof}

\begin{lem}\label{lem:image}
Suppose that $\alpha \neq 1$. 
Then for any $E\in D^b \Coh(S, \alpha)$, we have
\begin{align}\notag
&\Imm Z_G''(E) \subset \frac{\sqrt{3}}{4} \times \mathbb{Z}, 
\ \Ree Z_G''(E) \subset \frac{1}{4} \times \mathbb{Z}, \\
\notag
&\Ree Z_G''(E)- \frac{1}{\sqrt{3}} \Imm Z_G''(E)
\subset \frac{1}{2}\times \mathbb{Z}. 
\end{align}
\end{lem}
\begin{proof}
Let us write $\ch^B(E)=(\xi_0, \xi_1, \xi_2)$. 
By (\ref{v11}) and (\ref{rewrite}), we have 
\begin{align*}
&\Imm Z_G''(E)=
\frac{\sqrt{3}}{4} \times \left(\xi_1 h -\left( 3+\beta
 \right) \frac{\xi_0}{2}  \right) \\
& \Ree Z_G''(E)=-\left( \beta^2 + 3\beta h + 3 \right) \frac{\xi_0}{2}
+ \left( 3h+2\beta \right)\frac{\xi_1}{4} -\xi_2 \\
&\Ree Z_G''(E)- \frac{1}{\sqrt{3}} \Imm Z_G''(E)=
-(\beta^2 +4\beta h +6) \frac{\xi_0}{2}
+(h+\beta) \frac{\xi_1}{2} -\xi_2 
\end{align*}
Therefore the claim follows since $\xi_0=v_0^{\mathfrak{B}}(E)$ is 
even by Lemma~\ref{lem:tMu}. 
\end{proof}

\begin{lem}\label{lem:Fstable}
Suppose  that $\alpha \neq 1$. 
Then we have
\begin{align*}
F_S^{-1}(\oO_x) \in \aA_G
\end{align*}
for any $x\in S$, and it is $Z_G''$-semistable. 
\end{lem}
\begin{proof}
By the definition of $F_S$, we have 
\begin{align*}
F_S^{-1}(\oO_x) 
&\cong \Upsilon^{-1} \left\{ \mathrm{Cone}(\bB_1^{\oplus 2} \stackrel{\rm{ev}}{\twoheadrightarrow} L_{x})
  \right\} \otimes_{B_0}B_1[-1] \\
&\cong \Upsilon^{-1}\left\{ \mathrm{Ker} (\bB_2^{\oplus 2} \stackrel{\rm{ev}}{\twoheadrightarrow} L_{\iota(x)} ) 
\right\} \\
&\cong \mathrm{Ker}(\uU_2^{\oplus 2} \stackrel{\rm{ev}}{\twoheadrightarrow} 
\oO_{\iota(x)}).
\end{align*}
Note that every $\uU_i$ is 
$\mu$-stable, since there 
are no rank one subsheaves by Lemma~\ref{lem:tMu}.
The slope $\mu(\uU_2)$ can be easily 
computed to be $1/2$ by Lemma~\ref{lem:chU} and (\ref{v11}).  
Therefore 
both of $\uU_2$ and $F_S^{-1}(\oO_x)$ are
objects in $\aA_G$. 

Let us show that $\uU_2$ is $Z_G''$-stable. 
Using (\ref{chiU}), the complex number $Z_G''(\uU_2)$ is computed as
\begin{align*}
Z_G''(\uU_2)=-\frac{1}{4} + \frac{\sqrt{-3}}{4}. 
\end{align*}
By Lemma~\ref{lem:image}, it is enough to check the following: 
for any $E \in \aA_G$ with $\Imm Z_G''(E)=0$, 
we have $\Hom(E, \uU_2)=0$. 
Since such $E$ 
is a successive extensions by 
objects of the form $\oO_x$ for $x\in S$
or $F[1]$ for $\mu$-stable $F\in \Coh(S, \alpha)$
with $\mu(F)=0$, the vanishing $\Hom(E, \uU_2)=0$
is obvious. 

We next show the $Z_G''$-semistability of $F_S^{-1}(\oO_x)$. 
We have 
\begin{align*}
Z_G''(F_S^{-1}(\oO_x))=\frac{1}{2} + \frac{\sqrt{-3}}{2}. 
\end{align*}
Let $E\in \aA_G$ be a subobject of 
$F_S^{-1}(\oO_x)$ in $\aA_G$.  
We need to check that 
\begin{align}\label{ineq:EF}
\arg Z_G''(E) \le \arg Z_G''(F_S^{-1}(\oO_x)).
\end{align}
Similarly to the above argument, 
the imaginary part of $Z_G''(E)$ should be positive. 
By Lemma~\ref{lem:image}, 
we have the two possibilities: $\Imm Z_G''(E)=\sqrt{3}/4$
or $\sqrt{3}/2$. 
In the latter case, the inequality (\ref{ineq:EF}) is obvious 
since $Z_G''(F_S^{-1}(\oO_x)/E)$ lies in the negative real line. 

Suppose that $\Imm Z_G''(E)= \Imm Z_G''(\uU_2)=\sqrt{3}/4$. 
Since there is an exact sequence in $\aA_G$
\begin{align}\label{exact:FUe}
0 \to F_S^{-1}(\oO_x) \to \uU_2^{\oplus 2} \stackrel{\rm{ev}}{\twoheadrightarrow} \oO_{\iota(x)} \to 0
\end{align}
and $\uU_2$ is $Z_G''$-stable, we obtain the inequality 
\begin{align}\label{former}
\arg Z_G''(E) \le \arg Z_G''(\uU_2)
\end{align}
Furthermore the equality holds in (\ref{former})
only if $E=\uU_2$. 
However the exact sequence (\ref{exact:FUe}) 
shows that $\Hom(\uU_2, F_S^{-1}(\oO_x))=0$, so this case is excluded. 
Therefore the inequality (\ref{former}) is strict, and   
Lemma~\ref{lem:image} shows that
\begin{align*}
\Ree Z_G''(E) \ge \Ree Z_G''(\uU_2) + \frac{1}{2} =\frac{1}{4}.
\end{align*}
The above inequality implies (\ref{ineq:EF}). 
\end{proof}

By Corollary~\ref{cor:shift} and 
Proposition~\ref{thm:int}, the result of Theorem~\ref{intro:mainthm}
follows from the following statement: 
\begin{thm}\label{thm:final}
Suppose that $\alpha \neq 1$. 
Then the pair
\begin{align}\label{sigmaG}
\sigma_G \cneq (Z_G'', \aA_G)
\end{align}
is a Gepner type stability condition
on $D^b \Coh(S, \alpha)$
with respect to $(F_S, 2/3)$. 
\end{thm}
\begin{proof}
We divide the proof into three steps. 
\begin{sstep}
Checking the axioms. 
\end{sstep}
The construction of the pair (\ref{sigmaG})
is similar to Bridgeland's one in~\cite[Section~6]{Brs2}, 
so almost the same argument in~\cite[Section~6]{Brs2}
is applied to show that $\sigma_G$ is a stability 
condition. 
(Also see~\cite[Lemma~5.4]{StMa}.)
We may take care of 
the non-integrality of the untwisted Mukai vector 
$(r, \Delta, n)=\ch(E) \sqrt{\td_S}$. 
This is not a matter, 
since we have
\begin{align*}
\Delta^2 -2rs=\langle v^B(E), v^B(E) \rangle.
\end{align*}
The above number is an even integer, 
so the argument of~\cite[Lemma~6.2]{Brs1}
is not affected. 
Therefore~\cite[Lemma~6.2]{Brs1}
shows that the pair (\ref{sigmaG})
satisfies (\ref{cond:1}) if
the following condition holds: 
for any spherical twisted sheaf $F \in \Coh(S, \alpha)$
with $\mu(F)=0$, we have the inequality
$\Ree Z_G''(F)>0$. 
Once this conditions is checked to satisfy, 
the Harder-Narasimhan property 
is proved along with the same 
argument of~\cite[Proposition~7.1]{Brs2}, since
the image of $Z_G''$ is discrete. 
The support property is an easy consequence of the
same argument in~\cite[Lemma~8.1]{Brs2}, so 
the detail is left to the reader. 
\begin{sstep}
Non-existence of certain spherical twisted sheaves. 
\end{sstep}
Suppose by a contradiction that there is 
a spherical twisted sheaf $F \in \Coh(S, \alpha)$
with $\mu(F)=0$ and $\Ree Z_G'(F) \le 0$. 
For simplicity, we write $v_i \cneq v_i^{\mathfrak{B}}(F)$. 
The spherical condition of $F$, together with the 
Riemann-Roch theorem and the Serre duality implies
\begin{align}\label{v:sphe}
v_1^2=2v_0 v_2-2. 
\end{align}
The condition $\mu(F)=0$ implies $v_1 \cdot h=0$. 
Hence by the Hodge index theorem, we have
\begin{align*}
0=(v_1 \cdot h)^2/h^2 \ge v_1^2 =2v_0v_2-2
\end{align*}
which implies $v_0 v_2 \le 1$. Combined
with $\Ree Z_G(F)=-v_2 + 3v_0/16 \le 0$, we obtain 
the inequalities
\begin{align}\label{316}
\frac{3}{16}v_0 \le v_2 \le \frac{1}{v_0}
\end{align}
which show $v_0^2 \le 16/3$. 
Therefore from Lemma~\ref{lem:tMu}, we have 
$v_0=2$. Also we have   
$3/8 \le v_2 \le 1/2$ from (\ref{316}), 
hence $v_2=3/8$ by Lemma~\ref{lem:tMu}. 
By substituting into (\ref{v:sphe}), we 
obtain $v_1^2=-1/2$. 
Now by Lemma~\ref{lem:tMu}, 
$v_1$ is written as $\gamma/2$
for some $\gamma \in H^{1, 1}(S, \mathbb{Z})$.
By the above argument, we have 
$\gamma\cdot h=0$ and $\gamma^2 =-2$. 
However the latter condition implies that 
$\gamma$ or $-\gamma$ is represented by an 
effective divisor by the Riemann-Roch theorem on $S$. 
This contradicts to the former condition, so 
the property (\ref{cond:1}) is proved.  

\begin{sstep}
Gepner type property. 
\end{sstep}
We denote by $\Stab(S, \alpha)$ the space of
numerical 
stability conditions on $D^b \Coh(S, \alpha)$, 
and by $\zZ$ the forgetting map (\ref{Z:forget})
\begin{align*}
\zZ \colon
\Stab(S, \alpha) \to N(S, \alpha)_{\mathbb{C}}^{\vee}.
\end{align*}
By the argument so far, 
 we have shown that $\sigma_G \in \Stab(S, \alpha)$.  
Let us consider the following stability condition
\begin{align*}
\sigma_G' \cneq (-2/3) \cdot F_{S\ast} \sigma_G. 
\end{align*}
If we write $\sigma_G=(Z_G'', \{\pP_G(\phi)\}_{\phi \in \mathbb{R}})$
as in (\ref{num:pair}), then $\sigma_G'$ is written as
\begin{align*}
\sigma_G'=(Z_G'', \{\pP_G'(\phi)\}_{\phi \in \mathbb{R}}), \ 
\pP_G'(\phi)=F_S \pP_G(\phi -2/3). 
\end{align*}
By Lemma~\ref{lem:Fstable}, we see that 
$\oO_x$ for any $x\in S$
 is $\sigma_G'$-semistable with phase one. 
Let $U\subset \Stab(S, \alpha)$
be the open subset in which 
$\oO_x$ for any $x\in S$
is stable with the same phase. 
It is easy to see that $\sigma_G \in U$
(cf.~\cite[Lemma~6.2]{Brs2}), and the  
above argument shows that $\sigma_G' \in \overline{U}$. 
Now we use the same argument of~\cite[Corollary~11.3]{Brs2}, 
showing that
any point in $U$ is determined by the image of $\zZ$. 
Since $\Stab(S, \alpha)$ is Hausdorff, 
and the map $\zZ$ is a local homeomorphism~\cite{Brs1}, it 
follows that $\sigma_G=\sigma_G'$. 
\end{proof}

\begin{rmk}\label{rmk:thomas}
By the SOD (\ref{SOD:kuz}) and 
the gluing method in~\cite[Proposition~3.3]{InPo}, 
it is possible to construct stability conditions 
on $D^b \Coh(X)$ 
for general cubic fourfolds $X$
containing a plane, 
from stability conditions on $\dD_X$,
e.g. those constructed in this subsection. 
This idea was used in~\cite[Corollary~3.8]{BMMS}
to construct stability conditions on cubic 3-folds. 
\end{rmk}

\section{Appendix~A: Chern characters on graded matrix factorizations}
In this section, we recall the Chern character theory 
on graded matrix factorizations by 
Polishchuk-Vaintrob~\cite{PoVa}, \cite{PoVa2}, and prove
Lemma~\ref{lem:later}. 
\subsection{Chern characters and the central charge}
Let $W$ be a homogeneous polynomial as in (\ref{def:A}). 
The Chern character map on $\HMF(W)$
takes its value in the Hochschild homology 
group of $\HMF(W)$, which we denote by $\mathrm{HH}_{\ast}(W)$: 
\begin{align*}
\ch \colon K(\HMF(W)) \to \mathrm{HH}_{\ast}(W). 
\end{align*}
The above Chern character map is 
a composition of that of the $\mu_d$-equivariant 
matrix factorizations of $W$ with the forgetting the functor 
\begin{align*}
\HMF(W) \to \mathrm{HMF}^{\mu_d}(W).
\end{align*}
Here
$\mu_d$ acts on $x_i$ by weight $-1$. 
By~\cite[Theorem~2.6.1 (i)]{PoVa2}, 
$\mathrm{HH}_{\ast}(W)$ is described as 
\begin{align}\label{H:decom}
\mathrm{HH}_{\ast}(W) \cong \bigoplus_{\gamma \in \mu_d}
\mathrm{H}(W_{\gamma})^{\mu_d}. 
\end{align}
Here $H(W)$ is defined by 
\begin{align*}
H(W) \cneq \left(
\mathbb{C}[x_1, \cdots, x_n]/( \partial_{x_1}W, \cdots, \partial_{x_n}W) \right) dx_1 \wedge \cdots \wedge dx_n
\end{align*}
and the space $\mathrm{H}(W_{\gamma})$ is given
by applying the above construction for 
$W_{\gamma} \cneq W|_{(\mathbb{C}^n)^{\gamma}}$.
Note that we have 
\begin{align}\label{isom:mud}
H(W_{\gamma})^{\mu_d} \cong \mathbb{C}, \ 
\gamma \neq 1. 
\end{align} 
Since $\tau^{\times d}=[2]$ on $\HMF(W)$, we have 
$\tau_{\ast}^{\times d}=\id$ on $\mathrm{HH}_{\ast}(W)$. 
This implies that $\tau_{\ast}$ generates the $\mathbb{Z}/d\mathbb{Z}$-action 
on $\mathrm{HH}_{\ast}(W)$. 
By~\cite[Theorem~2.6.1 (ii)]{PoVa2}, 
the decomposition (\ref{H:decom})
coincides with the character decomposition of $\mathrm{HH}_{\ast}(W)$
with respect to the above $\mathbb{Z}/d\mathbb{Z}$-action. 

For $P^{\bullet} \in \HMF(W)$ and 
$0\le j\le d-1$, we denote by 
$\ch_{j}(P^{\bullet})$ the $H(W_{e^{2\pi j\sqrt{-1}/d}})$-component of  
$\ch(P^{\bullet})$ under the isomorphism (\ref{H:decom}). 
By~\cite[Theorem~3.3]{PoVa},
we have
\begin{align*}
\ch_0(P^{\bullet})= \mathrm{str}(\partial_{x_n} \delta_{P^{\bullet}}
 \circ \cdots \circ \partial_{x_1} \delta_{P^{\bullet}}).  
\end{align*} 
Here for a graded matrix factorization (\ref{MF}), 
the matrix $\delta_{P^{\bullet}}$ is given by
\begin{align*}
\delta_{P^{\bullet}} 
\cneq \left( \begin{array}{cc}
0 & p^0 \\
p^1 & 0
\end{array}  \right) \colon P^{\bullet} \to P^{\bullet}. 
\end{align*}
Note that $\ch_0(P^{\bullet})$ is always zero if $n$
is an odd integer. 

For $1\le j\le d-1$, we have (cf.~\cite[Theorem~3.3]{PoVa})
\begin{align*}
\ch_{j}(P^{\bullet})= \mathrm{str}(e^{2\pi j \sqrt{-1}/d} \colon 
P^{\bullet} \to P^{\bullet}). 
\end{align*}
In particular, the central charge $Z_G(P^{\bullet})$
 coincides with
$\ch_1(P^{\bullet})$. 

\subsection{Some computation of the central charge}
We set $R=A/(W)$ and 
$D_{\rm{sg}}^{\rm{gr}}(R)$ the triangulated 
category of singularities in the sense of~\cite{Orsin}. 
Namely $D_{\rm{sg}}^{\rm{gr}}(R)$ is the quotient 
category of the bounded derived category of finitely generated 
graded $R$-modules by the subcategory generated by finitely 
generated projective graded
$R$-modules. 
Then by~\cite[Theorem~3.10]{Orsin}, 
there is an equivalence of triangulated categories
\begin{align}\label{Cok}
\HMF(W) \stackrel{\sim}{\to} D_{\rm{sg}}^{\rm{gr}}(R)
\end{align}
sending a graded matrix factorization (\ref{MF})
to the cokernel of $p^0$. 
Let $\mathbb{C}(k)$ be the graded $R=A/(W)$-module 
given by 
\begin{align}\label{C(k)}
\mathbb{C}(k) \cneq (A/{\bf m})(k), \ {\bf m}=(x_1, \cdots, x_n) \subset A.
\end{align}
The object $\mathbb{C}(k)$ is regarded as an object in 
$D_{\rm{sg}}^{\rm{gr}}(R)$. 
By an abuse of notation, we denote by $\mathbb{C}(k)$ the 
corresponding object in $\HMF(W)$
under the equivalence (\ref{Cok}). 

\begin{lem}\label{lem:com}
We have $\ch_0(\mathbb{C}(k))=0$, and 
\begin{align*}
\ch_j(\mathbb{C}(k))= -e^{2\pi kj \sqrt{-1}/d}(1-e^{-2\pi j \sqrt{-1}/d})^n, \ 
1\le j\le d-1. 
\end{align*}
\end{lem}
\begin{proof}
Since (\ref{H:decom}) is the character decomposition with respect to 
the $\mathbb{Z}/d\mathbb{Z}$-action on $\mathrm{HH}_{\ast}(W)$
 generated by $\tau_{\ast}$, we have 
\begin{align*}
\ch_j(\mathbb{C}(k))=e^{2\pi kj \sqrt{-1}/d} \ch_j(\mathbb{C}(0)). 
\end{align*}
Therefore we may assume that $k=0$. 
The computation of $\ch_1(\mathbb{C}(0))=Z_G(\mathbb{C}(0))$
is given in~\cite[Example~2.8]{TGep}, and the same computation 
is applied for $\ch_j(\mathbb{C}(0))$ for $1\le j\le d-1$. 
It remains to prove that $\ch_0(\mathbb{C}(0))=0$. 
This is obvious if $n$ is odd, so we may assume that $n$ is 
even. Let $W'$ be 
\begin{align*}
W'=W(x_1, \cdots, x_{n-1}, 0) \in A' \cneq \mathbb{C}[x_1, \cdots, x_{n-1}]
\end{align*}
and set $R'=A'/(W')$. 
There is a natural push-forward functor 
\begin{align*}
i_{\ast} \colon 
D_{\rm{sg}}^{\rm{gr}}(R') \to D_{\rm{sg}}^{\rm{gr}}(R)
\end{align*}
by regarding a graded $R'$-module as a graded $R$-module 
by the surjection $R \twoheadrightarrow R'$. 
Combined with the equivalence (\ref{Cok})
and the functoriality of Hochschild homologies, 
we have the commutative diagram (cf.~\cite[Lemma~1.3.2]{PoVa})
\begin{align*}
\xymatrix{
\HMF(W') \ar[r]^{i_{\ast}} \ar[d]_{\ch} & \HMF(W) \ar[d]^{\ch}  \\
\mathrm{HH}_{\ast}(W') \ar[r]_{i_{H\ast}} & \mathrm{HH}_{\ast}(W). 
}
\end{align*}
Let $\tau'$ be the grade shift functor on $\HMF(W')$. 
Because $i_{\ast}$ commutes with $\tau'$ and $\tau$, 
the morphism $i_{H\ast}$ commutes with $\tau_{\ast}'$
and 
$\tau_{\ast}$ on $\mathrm{HH}_{\ast}(W')$ and $\mathrm{HH}_{\ast}(W)$
respectively (cf.~\cite[Lemma~1.2.1]{PoVa}).
Therefore $i_{H\ast}$ preserves the direct sum 
decomposition (\ref{H:decom}). 
Since $i_{\ast}\mathbb{C}(0)=\mathbb{C}(0)$, the 
vanishing $\ch_0(\mathbb{C}(0))=0$ follows from the 
case of $W'$. 
\end{proof}

\subsection{Proof of Lemma~\ref{lem:later}}
\begin{proof}
It is enough to show that, if $P^{\bullet} \in K(\HMF(W))$
satisfies $\chi(P^{\bullet}, Q^{\bullet})=0$ for any 
$Q^{\bullet} \in \HMF(W)$, then we have $\ch_1(P^{\bullet})=0$. 
Applying the Hirzebruch-Riemann-Roch theorem for matrix 
factorizatons~\cite[Theorem~4.2.1]{PoVa}, 
we have
\begin{align*}
&\chi(P^{\bullet}, \mathbb{C}(k)) \\
&=
\frac{1}{d} \left( 
\langle \ch_0(P^{\bullet}), \ch_0(\mathbb{C}(k)) \rangle_{W} +
\sum_{j=1}^{d-1}
c_{\lambda^j} \ch_j(P^{\bullet}) \cdot \ch_{-j}(\mathbb{C}(k)) \right). 
\end{align*} 
Here 
$\langle \ast, \ast \rangle_{W}$ is the Residue pairing on 
$H(W)$ (cf.~\cite[Proposition~4.1.2]{PoVa}),
$\lambda=e^{-2\pi \sqrt{-1}/d}$ and 
$c_{\gamma}$ for $\gamma \in \mu_d$ is given by 
\begin{align*}
c_{\gamma} \cneq 
\det (1-\gamma \colon {\bf m}/{\bf m}^2 \to {\bf m}/{\bf m}^2)^{-1}
\end{align*}
where ${\bf m}$ is given by (\ref{C(k)}). 
Combined with the 
assumption $\chi(P^{\bullet}, \mathbb{C}(k))=0$
and  
the 
computation of $\ch_j(\mathbb{C}(k))$
in Lemma~\ref{lem:com}, we obtain
\begin{align*}
\sum_{j=1}^{d-1} \lambda^{kj} \ch_j(P^{\bullet})=0
\end{align*}
for all $k\in \mathbb{Z}$. Then 
we have $\ch_j(P^{\bullet})=0$ for all $1\le j\le d-1$ by 
\begin{align*}
\det \left( \begin{array}{cccc}
1 & 1 & \cdots & 1 \\
\lambda & \lambda^2 & \cdots & \lambda^{d-1} \\
\vdots & \vdots & \ddots & \vdots \\
\lambda^{d-2} & \lambda^{2d-2} & \cdots & \lambda^{(d-2)(d-1)}
\end{array}  \right) 
= \prod_{1\le i< j\le d-1} (\lambda^{j} - \lambda^{i}) \neq 0. 
\end{align*}
The latter statement 
follows from the obvious condition 
\begin{align*}
Z_G(\tau P^{\bullet}) =e^{2\pi \sqrt{-1}/d} Z_G(P^{\bullet})
\end{align*}
for any $P^{\bullet} \in \HMF(W)$. 
\end{proof}

\section{Appendix~B: the lower dimensional cases}\label{sec:4}
In this section, we prove Theorem~\ref{intro:thm}. 
The case of $n\le 3$ is treated in~\cite{TGep}, so 
we assume $n=4$ or $5$. 
\subsection{t-structures on non-commutative surfaces}\label{subsec:slope}
We discuss t-structures on non-commutative surfaces and their
semiorthogonal summand. 
The situation here is applied both in $n=4$ and $5$ cases. 
Let $S$ be a smooth projective surface and 
$\bB_S$ a sheaf of $\oO_S$-algebras on
$S$, which is coherent as $\oO_S$-module. 
For an ample divisor $H$ in $S$, 
it defines the slope stability 
on $\Coh(\bB_S)$ by setting 
\begin{align}\notag
\mu_H(E)=\frac{c_1(\mathrm{Forg}(E)) \cdot H}{\rank \mathrm{Forg}(E)}. 
\end{align}
Here $\mathrm{Forg} \colon \Coh(\bB_S) \to \Coh(S)$
is forgetting the $\bB_S$-module structure, 
and $\mu_H(E)=\infty$ if $\rank \mathrm{Forg}(E)=0$. 
We put 
the following additional 
assumptions: 
\begin{itemize}
\item The object $\bB_S \in D^b \Coh(\bB_S)$ is $\mu_H$-stable 
and exceptional. 
\item There is a Serre functor $\sS_{\bB}$ on $D^b \Coh(\bB_S)$
given by 
\begin{align*}
\sS_{\bB}(E)= E \otimes_{\bB_S} M[2]
\end{align*}
for some $\bB_S$-bimodule $M$. 
\item For any $\mu_H$-semistable object $E \in \Coh(\bB_S)$, 
the object $E \otimes_{\bB_S} M$ is also $\mu_H$-semistable 
with $\mu_H(E \otimes_{\bB_S} M) < \mu_H(E)$. 
\end{itemize}
We
define the pair of full subcategories $(\tT, \fF)$
in $\Coh(\bB_S)$ as follows: 
\begin{align*}
\tT &\cneq \langle E :  \mu_H \mbox{-semistable with }
\mu_H(E)>\mu_H(\bB_S) \rangle_{\rm{ex}},  \\
\fF & \cneq \langle E : \mu_H \mbox{-semistable with }
\mu_H(E) \le \mu_H(\bB_S)
\rangle_{\rm{ex}}.
\end{align*}
The
 pair of subcategories $(\tT, \fF)$ forms a torsion pair in 
$\Coh(\bB_S)$, and its tilting 
is defined to be
\begin{align*}
\aA \cneq \langle \fF[1], 
\tT \rangle_{\rm{ex}}
\subset D^b \Coh(\bB_S). 
\end{align*}
The category $\aA$ is the heart of a bounded t-structure
on $D^b \Coh(\bB_S)$.
By our assumption, we have 
$\bB_S[1] \in \aA$. 
Our assumption that $\bB_S$ is exceptional 
allows us to define the triangulated 
category $\dD$ by the SOD
\begin{align}\label{defi:D}
D^b \Coh(\bB_S)= \langle \bB_S, \dD \rangle. 
\end{align}
The following lemma is a generalization of~\cite[Lemma~3.4]{BMMS}. 
\begin{lem}\label{lem:inter}
The intersection
$\cC \cneq \aA \cap \dD$
is the heart of a bounded t-structure on $\dD$. 
\end{lem}
\begin{proof}
Let us take $E \in \dD$. We denote by 
$\hH_{\aA}^i(E) \in D^b \Coh(\bB_S)$ the $i$-th 
cohomology of $E$ with respect to the t-structure 
with heart $\aA$. 
It is enough to show that $\hH_{\aA}^i(E) \in \dD$, i.e. 
$\dR \Hom(\hH_{\aA}^i(E), \bB_S)=0$ for all $i$. 
To show this, we see the following: 
for any $P \in \aA$, we have 
\begin{align}\label{vanish}
\Hom(P, \bB_S[j])=0, \ \mbox{ unless } j=1, 2. 
\end{align}
Indeed if (\ref{vanish}) holds, then the 
spectral sequence 
\begin{align}\label{spectral}
E_2^{p, q}= \Hom(\hH_{\aA}^{-q}(E), \bB_S[p]) \Rightarrow 
\Ext^{p+q}(E, \bB_S)=0
\end{align}
degenerates 
and we conclude that $\hH_{\aA}^i(E) \in \dD$
for all $i$. 

Let us prove (\ref{vanish}). We have the exact sequence 
in $\aA$
\begin{align*}
0 \to \hH^{-1}(P)[1] \to P \to \hH^0(P) \to 0. 
\end{align*}
Applying $\Hom(\ast, \bB_S[j])$ and noting that 
$\Hom(\tT, \bB_S)=0$ because $\bB_S \in \fF$, we see that 
(\ref{vanish}) holds for $j\le 0$. 
On the other hand, for $j\ge 3$, we use the duality
\begin{align}\label{vanish:2}
\Hom(P, \bB_S[j]) \cong \Hom(\bB_S[j-2], P\otimes_{\bB_S}M)^{\vee}
\end{align}
and show that the RHS is zero. 
Indeed we have the distinguished triangle
\begin{align*}
\hH^{-1}(P)\otimes_{\bB_S}M[1] \to P\otimes_{\bB_S} M \to 
\hH^0(P) \otimes_{\bB_S}M. 
\end{align*}
By our assumption, every $\mu_H$-semistable factor of 
$\hH^{-1}(P)\otimes_{\bB_S}M$ has $\mu_H$-slope less than $\mu_H(\bB_S)$. 
Therefore we have 
\begin{align*}
\Hom(\bB_S, \hH^{-1}(P)\otimes_{\bB_S}M)=0,
\end{align*}
 and 
this implies the vanishing of the RHS of (\ref{vanish:2}) for $j\ge 3$. 
\end{proof}
Let
$\Pi \colon D^b \Coh(\bB_S) \to \dD$
be the right adjoint functor of the inclusion
$i \colon \dD \hookrightarrow D^b \Coh(\bB_S)$. 
We have the following another description of $\bB$: 
\begin{lem}\label{lem:proj}
We have $\cC=\Pi(\aA)$. 
\end{lem}
\begin{proof}
The inclusion $\cC \subset \Pi(\aA)$ is obvious. 
Let $(\aA^{\le 0}, \aA^{\ge 0})$ be 
the t-structure on $D^b \Coh(\bB_S)$ with 
heart $\aA$. We show that the pair
$(\Pi \aA^{\le 0}, \Pi \aA^{\ge 0})$
is a bounded t-structure on $\dD$. Indeed if this 
is true, then we have
\begin{align}\label{inc}
\cC \subset \Pi(\aA) \subset \Pi \aA^{\le 0} \cap \Pi \aA^{\ge 0}
\end{align}
hence we obtain
$\cC=\Pi(\aA)$ as both sides
are hearts of bounded t-structures on $\dD$. 
In order to prove the above claim, the only 
thing to check is the vanishing
$\Hom(\Pi(P), \Pi(Q))=0$ for any 
$P \in \aA^{<0}$ and $Q \in \aA^{\ge 0}$. 
We have the distinguished triangle
\begin{align*}
i \circ \Pi(P) \to P \to P'
\end{align*}
with $P' \in \langle \bB_S \rangle$. 
Using the same notation in the proof of the previous lemma, we have
the following exact sequence in $\aA$
\begin{align}\label{ex:seq2}
\cdots \to \hH_{\aA}^{-1}(P') \to \hH_{\aA}^{0}(i\circ \Pi(P)) \to \hH_{\aA}^0(P) \cong 0 \to \cdots
\end{align}
and the isomorphism 
$\hH_{\aA}^{j}(P') \cong \hH_{\aA}^{j+1}(i\circ \Pi(P))$
for all $j\ge 0$. 
Because $P' \in \langle \bB_S \rangle$, we have 
$\hH_{\aA}^{j}(P') \cong \bB_S[1]^{\oplus m_j}$ for some 
$m_j \in \mathbb{Z}$. 
Combined with (\ref{ex:seq2}) and noting that 
$\bB_S[1] \in \aA$ is a simple object, which is easy to check,
we have 
$\hH_{\aA}^{j}(i\circ \Pi(P)) \in \langle \bB_S \rangle$
for $j\ge 0$.
This implies that $i\circ \Pi(P) \in \aA^{<0}$
since $\dR \Hom(i\circ \Pi(P), \bB_S)=0$. 
Therefore we have 
\begin{align*}
\Hom(\Pi(P), \Pi(Q)) &\cong \Hom(i\circ \Pi(P), Q) \\
&\cong 0. 
\end{align*}
\end{proof}
Let $\sS_{\dD}$ be the Serre functor of $\dD$. 
The following lemma will be useful 
in checking the Gepner type property. 
\begin{lem}\label{lem:check}
The subcategory $\sS_{\dD}(\cC)[-1] \subset \dD$ is 
obtained as a tilting of $\cC$. 
\end{lem}
\begin{proof}
The Serre functor $\sS_{\dD}$ is related to 
the Serre functor $\sS_{\bB}$ of $D^b \Coh(\bB_S)$ by
$\sS_{\dD}= \Pi \circ \sS_{\bB}$.
Therefore we have 
\begin{align}\notag
\sS_{\dD}(\cC)[-1] &= \Pi (\cC \otimes_{\bB_S}M)[1] \\
\notag
&\subset \Pi (\aA \otimes_{\bB_S} M)[1] \\
\label{check:11}
&\subset \Pi \langle \aA, \aA[-1] \rangle_{\rm{ex}}[1] \\
\label{check:12}
&\subset \langle \cC[1], \cC \rangle_{\rm{ex}}. 
\end{align}
Here (\ref{check:11}) follows from the assumption that 
$\ast \otimes_{\bB_S}M$ preserves the $\mu_H$-stability and 
decreases the $\mu_H$-slope, 
and (\ref{check:12}) follows from Lemma~\ref{lem:proj}. 
Therefore we obtain the assertion. 
\end{proof}

\subsection{The case of cubic surfaces}
In this subsection, we prove Theorem~\ref{intro:thm} for $n=4$. 
In the setting of the previous subsection, 
we set $S =(W=0)\subset \mathbb{P}^3$ to be 
the cubic surface, $\bB_S=\oO_S$
and $H$ is the hyperplane class. 
Since $\omega_S=\oO_S(-H)$, it satisfies the assumption 
in the previous subsection. 
Below we use the same notation in the previous 
subsection in the above setting. 

The cubic surface $S$ is a blow-up 
of $\mathbb{P}^2$ at six points
$\pi \colon S \to \mathbb{P}^2$. 
We denote by $C_1, \cdots, C_6$ the exceptional 
curves of $\pi$, and $h$ the hyperplane of $\mathbb{P}^2$
pulled-back to $S$. 
By~\cite{B-O2}, there is a full strong exceptional 
collection on $D^b \Coh(S)$
\begin{align*}
D^b \Coh(S)= \langle \oO_S, \oO_S(h), \oO_S(2h), \oO_{C_1}, \cdots, \oO_{C_6} \rangle. 
\end{align*}
Therefore the semiorthogonal summand $\dD$ in (\ref{defi:D})
has the SOD
\begin{align}\label{SOD:D}
\dD= \langle \oO_S(h), \oO_S(2h), \oO_{C_1}, \cdots, \oO_{C_6} \rangle
\subset D^b \Coh(S).
\end{align}
By
Orlov's theorem (\ref{Or:SOD}), 
the functor $\Phi_0$ gives an equivalence
\begin{align*}
\Phi_0 \colon \HMF(W) \stackrel{\sim}{\to}
\dD. 
\end{align*}
On the other hand, 
by (\ref{Serre}) the Serre functor of
$\HMF(W)$ is given by 
$\sS_W=\tau^{-1}[2]$. 
Since the Serre functors 
are categorical, $\sS_{\dD}$ and $\sS_{W}$ commute with $\Phi_0$.
Hence it is
enough to construct a Gepner type stability 
condition on $\dD$ with respect to 
$(\sS_{\dD}^{-1}[2], 2/3)$. 

Let us compute the central 
charge $Z_G$ in terms of $\dD$. 
By the SOD (\ref{SOD:D}), 
the numerical Grothendieck group of $\dD$ 
decomposes as 
\begin{align*}
N(\dD) \cong \mathbb{Z}[\oO_S(h)] \oplus 
\mathbb{Z}[\oO_S(2h)] \oplus \bigoplus_{i=1}^{6}
\mathbb{Z}[\oO_{C_i}]. 
\end{align*}
By Lemma~\ref{lem:later}, 
the central charge $Z_G' \cneq Z_G \circ \Phi_0^{-1}$ 
on $\dD$ is given by
\begin{align*}
Z_G'(E) =\chi(u, E)
\end{align*}
for some $u\in N(\dD)_{\mathbb{C}}$
satisfying $\sS_{\dD \ast}u=\omega \cdot u$
for $\omega=e^{2\pi \sqrt{-1}/3}$. 
Below we compute $u$ by looking 
at the action of $\sS_{\dD}$ on $N(\dD)$. 
Recall that the Serre functor $\sS_{\dD}$
on $\dD$ and the Serre functor $\sS_{\bB}=\otimes \oO_S(-H)[2]$
on $D^b \Coh(S)$
are related by $\sS_{\dD} = \Pi \circ \sS_{\bB}$. 
Hence there is a distinguished triangle for any $E\in \dD$
\begin{align}\label{dist:Serre}
\sS_{\dD}(E) \to E(-H)[2] \to \dR \Hom(E(-H)[2], \oO_S)^{\vee} \otimes \oO_S.
\end{align}
\begin{lem}
We have the following identities in $N(\dD)$: 
\begin{align}\notag
[\sS_{\dD}(\oO_S(h))] &=4[\oO_S(h)]-3[\oO_S(2h)]
+ \sum_{i=1}^6 [\oO_{C_i}] \\
\notag
[\sS_{\dD}(\oO_S(2h))] &=9[\oO_S(h)]-5[\oO_S(2h)]
+ \sum_{i=1}^6 [\oO_{C_i}] \\
\notag
[\sS_{\dD}(\oO_{C_i})] &=
2[\oO_S(h)] - [\oO_S(2h)] +[\oO_{C_i}] 
\end{align}
\end{lem}
\begin{proof}
It is easy to compute the Chern characters on 
the LHS using (\ref{dist:Serre}) and the Riemann-Roch 
theorem. By comparing them 
with those on the RHS, we can easily check the result. 
\end{proof}
Let us write $u\in N(\dD)_{\mathbb{C}}$ as
\begin{align*}
u=x_1[\oO_S(h)] + x_2[\oO_S(2h)] + \sum_{i=1}^{6} y_i[\oO_{C_i}]
\end{align*}
for $x_1, x_2, y_i \in \mathbb{C}$. 
By the above lemma, 
the linear equation $S_{\dD \ast}u=\omega \cdot u$
has the one dimensional solution space, spanned by 
\begin{align*}
x_1=3\omega, \ x_2=-3(\omega+1), \ y_1= \cdots =y_6=\omega+2. 
\end{align*}
Let us set
\begin{align*}
u_0=3\omega[\oO_S(h)]-3(\omega+1)[\oO_S(2h)]
+ (\omega+2)\sum_{i=1}^6[\oO_{C_i}]. 
\end{align*}
For $E \in \dD$, we compute 
$\chi(u_0, E)$ by using the Riemann-Roch theorem on $S$. 
Because $\dR \Hom(E, \oO_S)=0$, we have 
$\chi(E, \oO_S)=0$, which implies the constraint
\begin{align}\label{constraint}
\ch_2(E)=\frac{1}{2}\ch_1(E) \cdot H -\ch_0(E). 
\end{align}
On the other hand, noting 
that $H=3h-\sum_{i=1}^{6}C_i$, we have 
\begin{align}\label{chu}
\ch(u_0)=
\left(-3, -(\omega+2)H, -\frac{3}{2}\omega \right). 
\end{align}
By (\ref{constraint}), (\ref{chu}) and noting
$\td_{S}=(1, H/2, 1)$, a Riemann-Roch computation shows that 
\begin{align*}
\chi(u_0, E)=3 \ch_0(E) + (\omega-1) \ch_1(E) \cdot H. 
\end{align*}
As a summary, we obtain the following: 
\begin{lem}
There is a constant 
$c \in \mathbb{C}^{\ast}$ such that 
$Z_G'=Z_G \circ \Phi_0^{-1}$ on $\dD$
is written as
\begin{align*}
Z_G'(E)=c \cdot \left( 3\ch_0(E)-\frac{3}{2}\ch_1(E) \cdot H
+ \frac{\sqrt{-3}}{2}\ch_1(E) \cdot H  \right). 
\end{align*}
\end{lem}
Now we consider the heart $\cC \subset \dD$
constructed in Lemma~\ref{lem:inter}. 
The following statement proves Theorem~\ref{intro:thm}
for $n=4$: 
\begin{thm}\label{thm:surface}
The pair 
$\sigma_G=(Z_G''\cneq Z_G'/c, \cC)$
is a Gepner type stability condition on 
$\dD$ with respect to 
$(\sS_{\dD}^{-1}[2], 2/3)$. 
\end{thm}
\begin{proof}
We first check that (\ref{cond:1}) holds.
Since $\mu_H(\oO_S)=0$, the construction 
of $\cC$ yields
\begin{align*}
\Imm Z_G''(E)=\frac{\sqrt{3}}{2} \ch_1(E) \cdot H \ge 0
\end{align*}
for any $0\neq E \in \cC$. 
Suppose that $\Imm Z_G''(E)=0$. Then 
$E$ is a successive extensions in $D^b \Coh(S)$
by objects
of the form $\oO_x$ for $x\in S$
or $F[1]$ for $\mu_H$-stable 
sheaf $F$ on $S$ with $\mu_H(F)=0$.
Since any zero dimensional sheaf is not an object in $\dD$, 
we have $\Ree Z_G''(E)=3\ch_0(E)<0$. This implies the 
condition (\ref{cond:1}). 
The Harder-Narasimhan property is proved along with the 
same argument of~\cite[Proposition~2.4]{Brs2}. 
The support property is easy to check, and left to the reader. 

We show that $\sigma_G$ is Gepner type 
with respect to $(\sS_{\dD}^{-1}[2], 2/3)$, or 
equivalently with 
respect to $(\sS_{\dD}[-2], -2/3)$. 
Note that the action of $(-2/3)$ on stability conditions 
changes the corresponding hearts of the t-structures
by tilting shifted by $[-1]$.
By Lemma~\ref{lem:check}, the heart 
$\sS_{\dD}(\aA_G)[-2]$ is a tilting of $\aA_G$ shifted by $[-1]$. 
Therefore the desired Gepner type property of $\sigma_G$
follows from Lemma~\ref{lem:tilting} below.
\end{proof}
We have used the following lemma, whose proof is 
available in~\cite[Lemma~4.11]{TGep}.
\begin{lem}\label{lem:tilting}
Let $\dD$ be a $\mathbb{C}$-linear triangulated
category satisfying (\ref{Homfin}), 
$\aA \subset \dD$ is the heart of a bounded t-structure on $\dD$, 
and $Z \colon N(\dD) \to \mathbb{C}$ a group homomorphism. 
Suppose that there are torsion pairs $(\tT_k, \fF_k)$, $k=1, 2$
on $\aA$ such that, for $\cC_k=\langle \fF_k[1], \tT_k \rangle_{\rm{ex}}$
the associated tilting, both of the pairs
$(Z, \cC_1)$ and $(Z, \cC_2)$
give numerical stability conditions. Then 
$\cC_1=\cC_2$. 
\end{lem}

\subsection{The case of cubic 3-folds}\label{sec:5}
In this subsection, we prove Theorem~\ref{intro:thm}
for $n=5$. 
In this case, the variety $X=(W=0) \subset \mathbb{P}^4$
is a cubic 3-fold. 
There is a SOD
\begin{align*}
D^b \Coh(X)=\langle \oO_X(-2), \oO_X(-1), \dD_X \rangle
\end{align*}
hence the functor $\Phi_1$ in (\ref{Or:SOD}) gives an equivalence
\begin{align*}
\Phi_1 \colon \HMF(W) \stackrel{\sim}{\to} \dD_X. 
\end{align*}
In~\cite{BMMS}, motivated by Kuznetsov's work~\cite{Kuz}, 
Bernardara-Macri-Mehrotra-Stellari 
described the triangulated category $\dD_X$ 
in terms of sheaves of 
Clifford algebras on $\mathbb{P}^2$. 
Let $\bB_0$ (resp.~$\bB_1$) be the even (resp.~odd)
parts of the sheaf of Clifford algebras as in~\cite[Section~2]{BMMS}, 
given as follows:
\begin{align*}
\bB_0 &=\oO_{\mathbb{P}^2} \oplus \oO_{\mathbb{P}^2}(-1) \oplus
\oO_{\mathbb{P}^2}(-2)^{\oplus 2} \\
\bB_1 &=\oO_{\mathbb{P}^2}^{\oplus 2} \oplus \oO_{\mathbb{P}^2}(-1) \oplus
\oO_{\mathbb{P}^2}(-2).
\end{align*}
In the setting of Subsection~\ref{subsec:slope}, 
we set $S=\mathbb{P}^2$, $H$ is the hyperplane
class and $\bB_S=\bB_0$. 
All the assumptions in Subsection~\ref{subsec:slope}
are checked to satisfy in~\cite[Section~2.3]{BMMS}. 
Also there is an equivalence of triangulated categories
(denoted by $(\sigma_{\ast} \circ \Phi')^{-1}$ in~\cite{BMMS})
\begin{align*}
\Psi \colon \dD_X \stackrel{\sim}{\to} \dD 
\end{align*}
where $\dD$ is defined by (\ref{defi:D}). 
By (\ref{Serre}), the Serre functor 
$\sS_W$ on $\HMF(W)$
is given by $\tau^{-2}[3]=\tau[1]$. Therefore, it is enough 
to construct a Gepner type stability condition on $\dD$
with respect to 
$(\sS_{\dD}[-1], 2/3)$. 

The numerical Grothendieck group of 
$\dD_X$ is computed in~\cite[Section~2]{BMMS}, which 
is rather simpler than the $n=4$ case. 
Let $l\subset X$ be a line. 
We have the isomorphism (cf.~\cite[Proposition~2.7]{BMMS})
\begin{align*}
N(\dD_X) \cong \mathbb{Z}[I_l] \oplus \mathbb{Z}[\sS_{\dD_X}(I_l)]
\end{align*}
where $I_l$ is the ideal sheaf of $l$, $\sS_{\dD_X}$ is the
Serre functor of $\dD_X$. The Euler pairing $\chi$ is given by 
\begin{align}\label{matrix}
\left( \begin{array}{cc}
\chi(I_l, I_l) & \chi(I_l, \sS_{\dD_X}(I_l)) \\
\chi(\sS_{\dD_X}(I_l), I_l) &
\chi(\sS_{\dD_X}(I_l), \sS_{\dD_X}(I_l))
\end{array}  \right)
= \left( \begin{array}{cc}
-1 & -1 \\
0 & -1
\end{array}  \right). 
\end{align}
Furthermore 
the computation in~\cite[Proposition~2.7]{BMMS}
also shows the identity in $N(\dD_X)$:
\begin{align*}
[\sS_{\dD_X}^{-1}(I_l)]= [I_l]-[\sS_{\dD_X}(I_l)]. 
\end{align*}
Let us write $u\in N(\dD_X)_{\mathbb{C}}$ as 
\begin{align*}
u=x_1[I_l] + x_2[\sS_{\dD_X}(I_l)]
\end{align*} 
for $x_1, x_2 \in \mathbb{C}$. 
By the above argument, 
the linear equation
$-\sS_{\dD_X \ast}^{-1} u=\omega \cdot u$
 has the one dimensional solution space spanned by 
$(x_1, x_2)=(\omega, 1)$. We set
\begin{align*}
u_0=\omega[I_l] + [\sS_{\dD_X}(I_l)]. 
\end{align*}
On the other hand, let us consider the composition
\begin{align*}
\phi \colon 
N(\dD_X) \stackrel{\Psi_{\ast}}{\cong} N(\dD)
\stackrel{\mathrm{Forg}_{\ast}}{\to} N(\mathbb{P}^2) 
\stackrel{(\rank, c_1)}{\to} \mathbb{Z}^{\oplus 2}. 
\end{align*}
By the computation in~\cite[Proposition~2.12]{BMMS}, 
we have 
\begin{align*}
\Psi_{\ast}[I_l]=[\bB_0]-[\bB_1], \ 
\Psi_{\ast}[\sS_{\dD_X}(I_l)]=2[\bB_0]-[\bB_{-1}]
\end{align*}
where $\bB_{-1} \cneq \bB_1(-1)$. 
Hence we have 
\begin{align*}
\phi([I_l])= (0, 2), \ \phi([\sS_{\dD_X}(I_l)])=(4, -3). 
\end{align*}
In particular, $\phi$ induces the isomorphism 
over $\mathbb{Q}$, 
$\phi \colon N(\dD_X)_{\mathbb{Q}} \stackrel{\cong}{\to}
\mathbb{Q}^{\oplus 2}$. The inverse $\phi^{-1}$ is given by 
\begin{align*}
\phi^{-1}(r, d)= \left( \frac{3}{8}r + \frac{d}{2} \right) [I_l]
+ \frac{1}{4} [ \sS_{\dD_X}(I_l)]. 
\end{align*}
Using (\ref{matrix}), a little computation shows that
\begin{align*}
\chi(\phi^{-1}(r_1, d_1), \phi^{-1}(r_2, d_2))
= -\frac{19}{64}r_1 r_2 -\frac{3}{16}r_1 d_2 -\frac{5}{16}r_2 d_1
-\frac{1}{4}d_1 d_2. 
\end{align*}
Now let us consider the central charge on $\dD$
\begin{align*}
Z_G' \cneq Z_G \circ \Phi_{1\ast}^{-1} \circ \Psi^{-1}_{\ast} \colon 
N(\dD) \to \mathbb{C}. 
\end{align*}
It differs from 
$E \mapsto \chi(\Psi_{\ast}u_0, E)$
by multiplying a non-zero scalar constant. 
Noting that $\phi(u_0)=(4, 2\omega -3)$, the above computation yields
\begin{align*}
\chi(\Psi_{\ast}u_0, E) =
-\frac{\sqrt{3}}{4} \left\{
-\frac{\sqrt{3}}{12}(r+4d) + \left( d+ \frac{5}{4}r \right) \sqrt{-1} \right\}
\end{align*}
if $(r, d)= (\rank \mathrm{Forg}_{}(E), c_1 \mathrm{Forg}_{}(E))$.
We have obtained the following: 
\begin{lem}\label{prop:Z5}
There is a non-zero constant $c\in \mathbb{C}^{\ast}$ such that 
$Z_G'=Z_G \circ \Phi_{1\ast}^{-1} \circ \Psi^{-1}_{\ast}$
is written as 
\begin{align*}
Z_G'(E)= c\cdot \left( -\frac{\sqrt{3}}{12}(r+4d) + \left( d+ \frac{5}{4}r \right) \sqrt{-1} \right)
\end{align*}
where $(r, d)= (\rank \mathrm{Forg}_{}(E), c_1 \mathrm{Forg}_{}(E))$. 
\end{lem}

Now we consider the heart $\cC \subset \dD$
constructed in Lemma~\ref{lem:inter}. 
The following statement proves Theorem~\ref{intro:thm} for $n=5$:
\begin{thm}\label{thm:3fold}
The pair 
$\sigma_G=(Z_G''\cneq Z_G'/c, \cC)$
is a Gepner type stability condition on 
$\dD$ with respect to 
$(\sS_{\dD}[-1], 2/3)$. 
\end{thm}
\begin{proof}
We first check that (\ref{cond:1}) holds.
Since $\mu_H(\bB_0)=-5/4$, 
the construction of $\cC$
yields
\begin{align*}
\Imm Z_G''(E)=
d-\mu_H(\bB_0)r \ge 0
\end{align*}
for any $0\neq E \in \cC$
with $(r, d)= (\rank \mathrm{Forg}_{}(E), c_1 \mathrm{Forg}_{}(E))$. 
Suppose that $\Imm Z_G''(E)=0$, 
i.e. $4d+5r=0$. 
 Then 
$E$ is a successive extensions 
by objects of the form $Q$
with $\mathrm{Forg}(Q)$ zero dimensional 
or $F[1]$ for a $\mu_H$-stable object
 $F \in \Coh(\bB_0)$
with $\mu_H(F)=-5/4$. 
Since the former object
is not an object in $\dD$,
we have $r<0$. 
By substituting $4d+5r=0$, we obtain 
\begin{align*}
\Ree Z_G''(E)= \frac{\sqrt{3}}{3}r <0. 
\end{align*}
This implies that $\sigma_G$ satisfies 
(\ref{cond:1}). 
The Harder-Narasimhan property is proved along with the 
same argument of~\cite[Proposition~2.4]{Brs2}, 
and the support property is easy to check.  

We show that $\sigma_G$ satisfies the 
desired Gepner type property. 
Note that the action of $(2/3)$ on stability conditions 
changes the corresponding hearts of the t-structures
by tilting.
By Lemma~\ref{lem:check}, the heart 
$\sS_{\dD}(\cC)[-1]$ is a tilting of $\cC$. 
Therefore the desired Gepner type property of $\sigma_G$
follows from Lemma~\ref{lem:tilting}.
\end{proof}

\providecommand{\bysame}{\leavevmode\hbox to3em{\hrulefill}\thinspace}
\providecommand{\MR}{\relax\ifhmode\unskip\space\fi MR }
\providecommand{\MRhref}[2]{%
  \href{http://www.ams.org/mathscinet-getitem?mr=#1}{#2}
}
\providecommand{\href}[2]{#2}

Kavli Institute for the Physics and 
Mathematics of the Universe, University of Tokyo,
5-1-5 Kashiwanoha, Kashiwa, 277-8583, Japan.

\textit{E-mail address}: yukinobu.toda@ipmu.jp


\begin{thebibliography}{BMMS12}

\bibitem[AT]{AT}
N.~Addington and R.~Thomas, \emph{Hodge theory and derived categories of cubic
  fourfolds}, preprint, arXiv:1211.3758.

\bibitem[BD85]{BeDo}
A.~Beauville and R.~Donagi, \emph{La vari\'et\'e des droites d'une hypersurface
  cubique de dimension 4}, C.~R.~Acad.~Sci.~Paris.~S\'er.~I Math.~ \textbf{301}
  (1985), 703--706.

\bibitem[BFK12]{BFK}
M.~Ballard, D.~Favero, and L.~Katzarkov, \emph{Orlov spectra: bounds and gaps},
  Invent.~Math \textbf{189} (2012), 359--430.

\bibitem[BMMS12]{BMMS}
M.~Bernardara, E.~Macri, S.~Mehrotra, and P.~Stellari, \emph{A categorical
  invariant for cubic threefolds}, Adv.~Math.~ \textbf{229} (2012), 770--803.

\bibitem[BO]{B-O2}
A.~Bondal and D.~Orlov, \emph{Semiorthgonal decomposition for algebraic
  varieties}, preprint, arXiv:9506012.

\bibitem[Bri99]{Br2}
T.~Bridgeland, \emph{Equivalences of triangulated categories and
  {F}ourier-{M}ukai transforms}, Bull.~London Math.~Soc.~ \textbf{31} (1999),
  25--34.

\bibitem[Bri07]{Brs1}
\bysame, \emph{Stability conditions on triangulated categories}, Ann.~of Math
  \textbf{166} (2007), 317--345.

\bibitem[Bri08]{Brs2}
T.~Bridgeland, \emph{Stability conditions on ${K}$3 surfaces}, Duke Math.~J.~
  \textbf{141} (2008), 241--291.

\bibitem[CS07]{Ca-St}
A.~Canonaco and P.~Stellari, \emph{Twisted {F}ourier-{M}ukai functors},
  Adv.~Math.~ \textbf{212} (2007), 484--503.

\bibitem[CP10]{InPo}
J.~Collins and A.~Polishchuk,
\emph{Gluing stability conditions},
Adv.~Theor.~Math.~Phys.~ \textbf{14} (2010), 563--607. 


\bibitem[FJR]{FJRW1}
H.~Fan, T.~J. Jarvis, and Y.~Ruan, \emph{The {W}itten equation and its virtual
  fundamental cycle}, preprint, arXiv:0712.4025.

\bibitem[Has00]{BH2}
B.~Hassett, \emph{Special cubic fourfolds}, Compos.~Math.~ \textbf{120} (2000),
  1--23.

\bibitem[HRS96]{HRS}
D.~Happel, I.~Reiten, and S.~O. Smal$\o$, \emph{Tilting in abelian categories
  and quasitilted algebras}, Mem.~Amer.~Math.~Soc, vol. 120, 1996.

\bibitem[HS05]{HuSt}
D.~Huybrechts and P.~Stellari, \emph{Equivalences of twisted {$K$}3 surfaces},
  Math.~Ann.~ \textbf{332} (2005), 901--936.

\bibitem[JS12]{JS}
D.~Joyce and Y.~Song, \emph{A theory of generalized {D}onaldson-{T}homas
  invariants}, Mem.~Amer.~Math.~Soc.~ \textbf{217} (2012).

\bibitem[Kn{\"o}87]{Knor}
H.~Kn{\"o}rrer, \emph{Cohen-{M}acauley modules on hypersurface singularities},
  Invent.~Math.~ \textbf{88} (1987), 153--164.

\bibitem[KS]{K-S}
M.~Kontsevich and Y.~Soibelman, \emph{Stability structures, motivic
  {D}onaldson-{T}homas invariants and cluster transformations}, preprint,
  arXiv:0811.2435.

\bibitem[KST07]{KST1}
H.~Kajiura, K.~Saito, and A.~Takahashi, \emph{Matrix factorizations and
  representations of quivers {II}. {T}ype {ADE} case}, Advances in Math
  \textbf{211} (2007), 327--362.

\bibitem[KST09]{KST2}
\bysame, \emph{Triangualted categories of matrix factorizations for regular
  systems of weights $\varepsilon =-1$}, Advances in Math \textbf{220} (2009),
  1602--1654.

\bibitem[Kuz08]{Kuz}
A.~Kuznetsov, \emph{Derived categories of quadric fibrations and intersections
  of quadrics}, Adv.~Math.~ \textbf{218} (2008), 1340--1369.

\bibitem[Kuz10]{Kuz2}
\bysame, \emph{Derived categories of cubic fourfolds}, in: Cohomological and
  geometric approaches to rationality problems, Progr.~Math.~ \textbf{282}
  (2010), 219--243.

\bibitem[LMS]{LMS}
M.~Lahoz, E.~Macri, and P.~Stellari, \emph{{ACM} bundles on cubic threefolds
  and fourfolds containing a plane}, preprint, arXiv:1303.6998.

\bibitem[Mor96]{Mor}
J.~Morgan, \emph{The {S}eiberg-{W}itten equations and applications to the
  topology of smooth four-manifolds}, Mathematical Notes, vol.~44, Princeton
  University Press, 1996.

\bibitem[Orl09]{Orsin}
D.~Orlov, \emph{Derived categories of coherent sheaves and triangulated
  categories of singularities}, Algebra, arithmetic, and geometry: in honor of
  Yu.~I.~Manin, Progr.~Math.~ \textbf{270} (2009), 503--531.

\bibitem[PV]{PoVa2}
A.~Polishchuk and A.~Vaintrob, \emph{Matrix factorizations and {C}ohomological
  {F}ield {T}heories}, preprint, arXiv:1105.2903.

\bibitem[PV12]{PoVa}
\bysame, \emph{Chern characters and {H}irzebruch-{R}iemann-{R}och formula for
  matrix factorizations}, Duke Math.~J.~ \textbf{161} (2012), 1863--1926.

\bibitem[SM12]{StMa}
P.~Stellari and E.~Macri, \emph{Fano varieties of cubic fourfold containing a
  plane}, Math.~Ann.~ \textbf{354} (2012), 1147--1176.

\bibitem[ST01]{ST}
P.~Seidel and R.~P. Thomas, \emph{Braid group actions on derived categories of
  coherent sheaves}, Duke Math.~J.~ \textbf{108} (2001), 37--107.

\bibitem[Tak]{Tak}
A.~Takahashi, \emph{Matrix factorizations and representations of quivers {I}},
  preprint, arXiv:0506347.

\bibitem[Toda]{TGep2}
Y.~Toda, \emph{Gepner point and strong {B}ogomolov-{G}ieseker inequality for
  quintic 3-folds}, preprint, arXiv:1305.0345.

\bibitem[Todb]{TGep}
\bysame, \emph{Gepner type stability conditions on graded matrix
  factorizations}, preprint, arXiv:1302.6293.

\bibitem[Voi86]{Voi}
C.~Voisin, \emph{Th\'eor\`eme de {T}orelli pour les cubiques de
  $\mathbb{P}^5$}, Invent.~Math.~ \textbf{86} (1986), 577--601.

\bibitem[Wal05]{Wal}
J.~Walcher, \emph{Stability of {L}andau-{G}inzburg branes}, Journal of
  Mathematical Physics \textbf{46} (2005), arXiv:hep-th/0412274.

\end{thebibliography}
\end{document}